\documentclass[10pt,a4]{article}
\usepackage{latexsym}
\usepackage{amsmath}
\usepackage{amssymb}
\usepackage{amsthm}
\usepackage{amscd}
\usepackage{mathrsfs}
\usepackage[all]{xy}

\newtheorem{theorem}{Theorem}[section]
\newtheorem{lemma}[theorem]{Lemma}
\newtheorem{corollary}[theorem]{Corollary}
\newtheorem{proposition}[theorem]{Proposition}

\theoremstyle{definition}

\newtheorem{remark}[theorem]{Remark}
\newtheorem{definition}[theorem]{Definition}

\theoremstyle{remark}
\newtheorem{notation}{Notation}

\makeatletter

\@addtoreset{equation}{section}
\makeatother

\renewcommand{\eqref}[1]{(\ref{#1})}

\renewcommand{\bigskip}{\vspace{0.2cm}}

\begin{document}

\title{Stable Reduction of $X_0(p^4)$}

\maketitle

\begin{center}
{\sc Takahiro Tsushima}\\
{\it
Graduate School of Mathematical Sciences,
The University of Tokyo, 3-8-1 Komaba, Meguro-ku
Tokyo 153-8914, JAPAN}
\\
E-mail: tsushima@ms.u-tokyo.ac.jp
\end{center}
\maketitle
\begin{abstract}
R.\ Coleman and K.\ McMurdy calculated 
the stable reduction of $X_0(p^3)\ (p \geq 13)$
 on the basis of the rigid geometry in \cite{CM}.    
In this paper, we determine
the stable model of $X_0(p^4)$ for primes $p \geq 13$ on the basis of their idea.
In the stable model of $X_0(p^4),$ a certain number of copies of the Deligne-Lusztig curve 
for ${\rm SL}_2(\mathbb{F}_p),$ defined by $a^p-a=t^{p+1},$ appear.
\end{abstract}

\section{Introduction}
\noindent
Let $n$ be an integer and $p$ a prime number.
It is known that if $n \geq 3$ and $p \geq 5,$
or if $n \geq 1$ and $p \geq 11,$
the modular curve $X_0(p^n)$
does not have a model with good reduction
over the ring of integers of any complete subfield of $\mathbb{C}_p.$
By a model for a scheme $X$ over a complete local field $K,$
we mean a scheme $\mathcal{X}$ over the ring of integers
$\mathcal{O}_K$ of $K$ such that $X \simeq \mathcal{X} \otimes_{\mathcal{O}_K} K.$ 
When a curve $C$ over $K$
does not have a model with good reduction
over $\mathcal{O}_K,$
it may have the ``next best thing,"
i.e., {\it a stable model}.
The stable model is unique up to isomorphism if it exists,
and it does over the ring of integers in
some finite extension of $K,$
as long as the genus of the curve is at least $2,$
which is proved by Deligne and Mumford in \cite{DM}.
Moreover, if $\mathcal{C}$
is a stable model for $C$
over $\mathcal{O}_K,$
and $K \subset L \subset \mathbb{C}_p,$
then $\mathcal{C} \otimes_{\mathcal{O}_K} \mathcal{O}_L$
is a stable model for $C \otimes_K L$ over
$\mathcal{O}_L.$
The special fiber of any stable model for $C$
is called the stable reduction.

The stable models of $X_0(p)$ and
$X_0(p^2)$ were previously known,
due to works of J-I.\ Igusa and Deligne-Rapoport \cite[Section 7.6]{DR},
 and B.\ Edixhoven \cite[Theorem 2.1.2]{E} respectively. 
In \cite{CM}, R.\ Coleman and K.\ McMurdy calculated 
the stable reduction of $X_0(p^3)$
on the basis of the rigid geometry. They use the Woods Hole theory in \cite{WH} 
and the Gross-Hopkins theory in \cite{GH}
 to analyze the supersingular locus. 
 They also determine 
 the stable reduction of $X_0(Np^3)$ with $(N,p)=1$
and compute the inertia action on the stable model
 of $X_0(p^3)$ in \cite{CM2}. 
See \cite[Introduction]{CM}
for other prior results regarding the stable models
of modular curves at prime power levels.
Using the type theory of Bushnell-Kutzko, 
J.\ Weinstein makes a conjecture on the stable model of the modular curve $X_0(p^n)$ in \cite{W}.

In this paper, we calculate 
the stable reduction of $X_0(p^4)$  on the basis of the ideas in \cite{CM}.
The approach of this paper is rigid analytic as in loc.\ cit.
Our strategy to find the stable model is the same as the one in loc.\ cit.
We construct a stable model of $X_0(p^4)$ by actually constructing
a stable covering by wide opens.
The concept of the stable covering is invented by Coleman to
compute the stable reduction of a curve over a local field.
See \cite[subsections 2.2 and 2.3]{CM} or \cite[Section 1]{CW}
for the notions of wide open space and stable covering.
The groundwork in a rigid analytic setting 
has been done in \cite[Section 2]{CM}. See also \cite{C1}.
A covering of the ordinary locus of the modular curve $X_0(p^n)$
can be obtained by extending the ordinary affinoids $\mathbf{X}^{\pm}_{a,b}$
 with $a+b=n,a \geq 0,b \geq 0$ defined in \cite{C2} to wide open neighborhoods $W^{\pm}_{a,b}.$
See also \cite[Section 13]{KM} for the treatment of the ordinary locus.
The ordinary regions of $X_0(p^2)$ and $X_0(p^3)$
are covered by four wide opens $W_{2,0},W^{\pm}_{1,1},W_{0,2}$
and six wide opens $W_{3,0},W^{\pm}_{2,1},W^{\pm}_{1,2},W_{0,3}$
respectively as in \cite[subsection 3.2, Theorem 5.3 and Theorem 9.2]{CM}.
Similarly as the stable coverings of $X_0(p^2)$ and $X_0(p^3),$
the ordinary region of $X_0(p^4)$
is covered by eight wide opens, which are denoted by
$W_{4,0}, W^{\pm}_{3,1},W^{\pm}_{2,2}, W^{\pm}_{1,3}$ and $W_{0,4}.$
See subsection 5.1 for the definitions of these spaces.
As explained above, these spaces contain affinoid subdomains $\mathbf{X}^{\pm}_{a,b}\ 
(a+b=4, a \geq 0, b \geq 0)$, whose reductions are known to be
the Igusa curves ${\rm Ig}(p^{{\rm min}(a,b)}).$ 
This fact is proved in \cite{C2}.
See also \cite[Proposition 3.6]{CM}. Let $\overline{W}^{\pm}_{a,b}$
denote the reduction of the space $W^{\pm}_{a,b}.$

The supersingular locus essentially breaks up into the union of finitely many deformation spaces of height $2$
formal groups with level structure.
We produce a covering of the supersingular locus on the basis of Coleman-McMurdy's
ideas in \cite{CM} and \cite{Mc}. Finally, we show
that the ``genus" of the covering is the genus of $X_0(p^4),$
and therefore that the overall covering is stable.
See subsection 5.3 for the genus computation. 
It is known that 
all of the connected components of the supersingular locus
of $X_0(p^n)$ are (nearly) isomorphic, which is proved in \cite[Proposition 4.2]{CM}.
Because of this fact, to compute the stable model
of the modular curve $X_0(p^n),$
it suffices to analyze the tubular neighborhood in $X_0(p^n)$
of one supersingular elliptic curve $A/\mathbb{F}_p.$
We denote by $W_A(p^n)$ the tubular neighborhood of $A$
in $X_0(p^n).$
We will analyze the 
supersingular locus $W_A(p^4)$ in detail and construct a stable covering of $X_0(p^4).$
Unlike $W_A(p^2),$ however,
$W_A(p^3)$ and $W_A(p^4)$ must itself 
be covered by smaller wide opens, because its reductions contain multiple irreducible components
as mentioned in \cite[Section 1.1]{CM}.

We define several rigid analytic subspaces of $W_A(p^n)$ on the basis of the idea
of \cite[subsection 5.1]{Mc},
whose reductions are expected to play a key role 
in the stable reductions
of the modular curves. 
We set $i(A):=|{\rm Aut}(A)|/2.$
The space $W_A(p)$ is known to be isomorphic to an annulus $A(p^{-i(A)},1).$
See Notation below for our rigid analytic notation.
We fix an isomorphism $W_A(p) \simeq A(p^{-i(A)},1)$ appropriately. See subsection 2.2
for the identification.
We focus on the circles $\mathbf{TS}_A:=C[p^{-\frac{pi(A)}{p+1}}] \subset W_A(p)$
and $\mathbf{SD}_A:=C[p^{-i(A)/2}] \subset W_A(p).$
Let $\pi_f,\pi_v:X_0(p^n) \longrightarrow X_0(p^{n-1})$ be level-lowering maps.
See Definition \ref{low} for these maps.
Let $a,b \in \mathbb{Z}_{\geq 0}.$ We put $\pi_{a,b}:=\pi_v^a \circ \pi_f^b.$
Let $n$ be a positive integer. Assume that $a,b$ are positive integers.
We define as follows
$$\mathbf{Y}^A_{a,b}:=\pi^{-1}_{a,b-1}(\mathbf{TS}_A) \subset W_A(p^n)$$
with $a+b=n \geq 2$ and 
$$\mathbf{Z}^A_{a,b}:=\pi_{a,b}^{-1}(\mathbf{SD}_A) \subset W_A(p^n)$$
with $a+b=n-1 \geq 2.$
See subsection 2.3 for more detail.
As mentioned above, 
the reductions of these affinoids
 are expected to play a fundamental role in the stable models of modular curves.
Let $\overline{\mathbf{Y}}^A_{a,b}$ and $\overline{\mathbf{Z}}^A_{a,b}$
denote the reductions of the spaces
${\mathbf{Y}}^A_{a,b}$ and ${\mathbf{Z}}^A_{a,b}$
respectively.
Actually, the reduction $\overline{\mathbf{Y}}^A_{1,1}$
appears in the stable reduction of $X_0(p^2)$ as in \cite{E},
 and the reductions $\overline{\mathbf{Y}}^A_{2,1}, \overline{\mathbf{Y}}^A_{1,2}$
  and $\overline{\mathbf{Z}}^A_{1,1}$ appear in the stable reduction of $X_0(p^3)$
  as in \cite{CM}.
The main parts of the works \cite{E} and \cite{CM} are in calculating
the reductions of ${\mathbf{Y}}^A_{1,1} \subset W_A(p^2)$ and 
$\mathbf{Z}^A_{1,1} \subset W_A(p^3)$ respectively.
In \cite[Lemma 5.1 and Proposition 7.1]{CM}, to calculate the reductions $\overline{\mathbf{Y}}^A_{1,1}$
and $\overline{\mathbf{Z}}^A_{1,1},$
 Coleman-McMurdy consider an embedding of them into a product of 
 the subspaces of $W_A(p),$ and apply de Shalit's
 approximation theorem for the map $\pi_f:W_A(p) \longrightarrow W_A(1).$
 See Theorem \ref{thm} for de Shalit's theorem. 

In this paper, to compute irreducible components which appear in the reduction of $W_A(p^4),$
we consider an embedding into a product of an open unit ball $W_A(1) \subset X_0(1)$ and 
use the {\it Kronecker polynomial.}
Actually, de Shalit's approximation theorem is deduced from the shape
of Kronecker's polynomial in \cite{dSh}.
The Kronecker polynomial $F_p(X,Y)=(X^p-Y)(X-Y^p)+pf(X,Y) \in \mathbb{Z}[X,Y]$ 
gives a plane model of $X_0(p)$ over $\mathbb{Q}.$
See subsection 2.1. 

In the following, we will explain the shape of the stable model of $X_0(p^4).$
As mentioned above, we fix a supersingular elliptic curve $A/\mathbb{F}_p$
with $j(A) \neq 0, 1728$ and analyze the locus $W_A(p^4).$ 
First of all, the reduction of $W_A(p^4)$
contains two isomorphic lifts 
$\overline{\mathbf{Y}}^A_{1,3}$ and $\overline{\mathbf{Y}}^A_{3,1}$
of a supersingular component $\overline{\mathbf{Y}}^A_{1,1}$ of $X_0(p^2),$
with each meeting exactly three 
of the ordinary components.
For example, the reduction $\overline{\mathbf{Y}}^A_{3,1}$ 
meets the reductions of $W_{4,0},W^{\pm}_{3,1}.$
The component $\overline{\mathbf{Y}}^A_{1,1}$ of $X_0(p^2)$ is the ``horizontal component''
found by Edixhoven in \cite[Theorem 2.1.2]{E}.
The curve  $\overline{\mathbf{Y}}^A_{1,1}$ is defined by the equation $$xy(x-y)^{p-1}=1$$
and its genus is equal to $(p-1)/2.$
Coleman-McMurdy give a rigid analytic interpretation to the horizontal component of Edixhoven in 
\cite[Proposition 5.2]{CM}.
See subsection 4.2 for more detail.
Furthermore, the reduction of $W_A(p^4)$
contains two isomorphic lifts 
$\overline{\mathbf{Z}}^A_{1,2}$ and $\overline{\mathbf{Z}}^A_{2,1}$
of a supersingular component $\overline{\mathbf{Z}}^A_{1,1}$ 
in the stable reduction of
 $X_0(p^3).$
The component $\overline{\mathbf{Z}}^A_{1,1}$ 
is found by Coleman-McMurdy in \cite[Proposition 8.2]{CM}, which they call
 the ``bridging component''.
The curve $\overline{\mathbf{Z}}^A_{1,1}$
is defined by $$Z^p+X^{p+1}+\frac{1}{X^{p+1}}=0$$ and its genus is equal to $0.$
This curve has $2(p+1)$ singular points at $X=\zeta$ with $\zeta^{2(p+1)}=1.$
Furthermore, in the stable reduction of $X_0(p^3),$
the component $\overline{\mathbf{Z}}^A_{1,1}$
 meets (in distinct points) a certain number
of isomorphic copies of a curve of genus $(p-1)/2$ defined by $a^p-a=s^2.$
This phenomenon for $X_0(p^3)$ is first observed by Coleman-McMurdy
in \cite{CM}.
See Section 3 for the computation of these components.
Similarly as above, in the stable reduction of $X_0(p^4),$
the components $\overline{\mathbf{Z}}^A_{1,2}$ and $\overline{\mathbf{Z}}^A_{2,1}$
 meet (in distinct points) a certain number
of isomorphic copies of a curve of genus $(p-1)/2$ defined by $a^p-a=s^2.$
In the stable reduction of $X_0(p^4),$ these two ``old" components
$\overline{\mathbf{Z}}^A_{1,2}$ and $\overline{\mathbf{Z}}^A_{2,1}$
 are connected through a central
component $\overline{\mathbf{Y}}^A_{2,2},$ which we call the ``new bridging component" in 
the stable reduction of $X_0(p^4).$
This curve $\overline{\mathbf{Y}}^A_{2,2}$
is defined by the following equations
\begin{equation}\label{kk}
xy(x-y)^{p-1}=1,Z^p+1+\frac{1}{x^{p+1}}+\frac{1}{y^{p+1}}=0,
\end{equation}
and its genus is equal to $(p-1)/2.$
The component $\overline{\mathbf{Y}}^A_{2,2}$
 meets exactly two ordinary components $\overline{W}_{2,2}^{\pm}.$
The curve (\ref{kk}) has $(p+1)$ singular points at $(x,y)=(-\zeta,\zeta)$
with $\zeta^{p+1}=-1.$ See Corollary \ref{good} for more detail.
To complete the picture, the new bridging component $\overline{\mathbf{Y}}^A_{2,2}$
then meets (in distinct points) a certain number
of isomorphic copies of a curve of genus $p(p-1)/2$, defined by $a^p-a=t^{p+1}.$
This curve is called the {\it Deligne-Lusztig curve for $SL_2(\mathbb{F}_p).$}
This is a new phenomenon that is observed in the stable reduction of $X_0(p^4).$
See Corollary \ref{op} for more detail.
The phenomenon for $X_0(p^4)$ is compatible with the Weinstein 
conjecture mentioned above.
In Section 4, all irreducible components explained above in the stable reduction
of $X_0(p^4)$ are calculated. 

In the table below, 
we give the following pairs $(x,y)$
\begin{itemize}
\item $x$ is the genus of the component  
$\overline{\mathbf{Z}}^A_{a,b}\ (a+b=3)$ or 
$\overline{\mathbf{Y}}_{a,b}^A\ (a+b=4).$ \\
\item $y$ is the number of copies of the curve 
$a^p-a=s^2$ or
$a^p-a=t^{p+1}$ which intersects the component
$\overline{\mathbf{Z}}^A_{a,b}\ (a+b=3)$ or $\overline{\mathbf{Y}}_{a,b}^A\ (a+b=4).$
\end{itemize}
\begin{center}
{\renewcommand\arraystretch{2}
\begin{tabular}{|c|c|c|c|}
\hline  & $\overline{\mathbf{Z}}^A_{1,2},
\overline{\mathbf{Z}}^A_{2,1}$ & $\overline{\mathbf{Y}}^A_{1,3},\overline{\mathbf{Y}}^A_{3,1}$ 
& $\overline{\mathbf{Y}}^A_{2,2}$ \\ 
\hline  $j(A)=0$ & $(0,\frac{2(p+1)}{3})$ & $(\frac{p-5}{6},0)$ & 
$(\frac{p-5}{6},\frac{p+1}{3})$ \\ 
\hline  $j(A)=1728$ & $(0,p+1)$ & $(\frac{p-3}{4},0)$ &  $(\frac{p-3}{4},\frac{p+1}{2})$\\ 
\hline  otherwise & $(0,2(p+1))$ & $(\frac{p-1}{2},0)$ &  $(\frac{p-1}{2},p+1)$\\ 
\hline 
\end{tabular}}
\end{center} 
\begin{center}
Table 1: Genera of Supersingular Components of $X_0(p^4)$ etc.
\end{center}
 Partial dual graphs of the stable reductions
of $X_0(p^n)\ (2 \leq n \leq 4),$
 including one complete supersingular region,
are given below.

\centerline{
\xygraph{
\circ ([]!{+(0,+.3)} {\overline{W}_{2,0}})
(-[l]\cdots)(-[r]\circ ([]!{+(0,+.3)} {\overline{W}^+_{1,1}})-[r]
\circ ([]!{+(0,+.3)} {\overline{\mathbf{Y}}^A_{1,1}})
(-[r]\circ ([]!{+(0,+.3)} {\overline{W}^{-}_{1,1}})
(-[r]\circ([]!{+(-.1,+.3)} {\overline{W}_{0.2}})(-[r] \cdots))))}}
\begin{center}
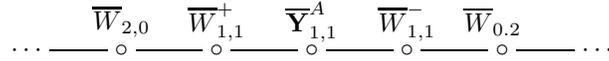
Figure 1: A partial dual graph of $X_0(p^2)$.
\end{center}
\centerline{
\xygraph{
\circ ([]!{+(0,+.3)} {\overline{W}^{-}_{2,1} })(-[l]\cdots)(-[r]
\circ ([]!{+(+.1,+.3)} {\overline{\mathbf{Y}}^A_{2,1}})
(-[dl] \circ ([]!{+(0,-.3)} {\overline{W}^+_{2,1}})(-[l] \cdots))
(-[ul] \circ
([]!{+(0,+.3)} {\overline{W}_{3,0}})(-[l] \cdots))(-[r]
\circ ([]!{+(0,+.3)} {\overline{\mathbf{Z}}^A_{1,1}})
(-[r]
\circ([]!{+(-.1,+.3)} {\overline{\mathbf{Y}}^A_{1,2}})
(-[r] \circ([]!{+(0,+.3)} {\overline{W}^{-}_{1,2}})(-[r] \cdots))
(-[ur] \circ([]!{+(0,+.3)} {\overline{W}_{0,3}})(-[r] \cdots)) 
(-[dr] \circ([]!{+(0,-.3)} {\overline{W}^{+}_{1,2}})(-[r] \cdots))))))}}
\begin{center}
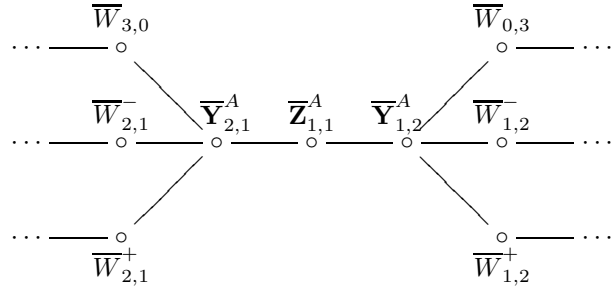
Figure 2: A partial dual graph of $X_0(p^3)$.
\end{center}

\centerline{
\xygraph{
\circ ([]!{+(0,+.3)} {\overline{W}^{-}_{3,1} })(-[l]\cdots)(-[r]
\circ ([]!{+(+.1,+.3)} {\overline{\mathbf{Y}}^A_{3,1}})(-[dl] \circ ([]!{+(0,-.3)} 
{\overline{W}^+_{3,1}})(-[l] \cdots))(-[ul] \circ
([]!{+(0,+.3)} {\overline{W}_{4,0}})(-[l] \cdots))(-[r]
\circ ([]!{+(0,+.3)} {\overline{\mathbf{Z}}^A_{2,1}})-[r]
\circ ([]!{-(+.4,+.3)} {\overline{\mathbf{Y}}^A_{2,2}})(-[d] \circ ([]!{+(-.3,-.3)} 
{\overline{W}^{-}_{2,2}})(-[d] \vdots))
(-[u] \circ([]!{+(+.3,+.3)} {\overline{W}^{+}_{2,2}})(-[u] \vdots))(-[r]
\circ ([]!{+(0,+.3)} {\overline{\mathbf{Z}}^A_{1,2}})
(-[r]
\circ([]!{+(-.1,+.3)} {\overline{\mathbf{Y}}^A_{1,3}})(-[r] \circ([]!{+(0,+.3)} 
{\overline{W}^{-}_{1,3}})(-[r] \cdots))
(-[ur] \circ([]!{+(0,+.3)} {\overline{W}_{0,4}})(-[r] \cdots)) 
(-[dr] \circ([]!{+(0,-.3)} {\overline{W}^{+}_{1,3}})(-[r] \cdots))))))}}
\begin{center}
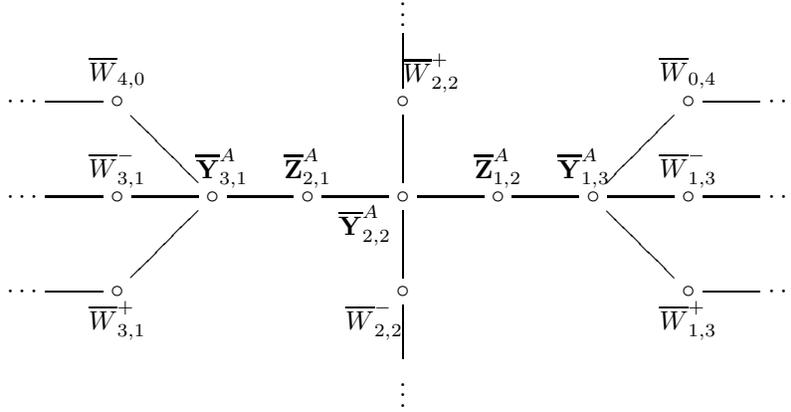
Figure 3: A partial dual graph of $X_0(p^4)$.
\end{center}
In section 5, we construct a covering of $X_0(p^4)$
and prove that it is actually the stable covering. See \cite[subsection 2.3]{CM}
for stable covering.
Each piece of the covering is a wide open space which contains the affinoids
$\mathbf{Y}^A_{a,b}\ (a+b=4)$ and $\mathbf{Z}^A_{a,b}\ (a+b=3)$ defined above
as the underlying affinoids. See \cite[subsection 2.2]{CM} for wide open space. 
Then, we have to show that these wide open spaces are basic.
In other words, it is proved that all intersections of wide open spaces of the covering
 are annuli.
 See loc.\ cit. for basic wide open space.
We also compute intersection multiplicities
of irreducible components
 in the stable reduction of $X_0(p^n)\ (n=2,3,4)$
in Tables 3, 4 and 5 in subsection 5.4.

In subsequent papers, we will investigate the stable models of the modular curve $X_0(p^5)$
and a certain Shimura curve.
We are greatly inspired by the works of Coleman-McCallum in \cite{CW} on the stable reduction
of the quotient of the Fermat curve and of Coleman-McMurdy on the stable reduction of the modular
 curve $X_0(p^3)$ in \cite{CM}. 
 We are also inspired by a conjecture for the stable reduction of $X_0(p^4)$ in \cite{Mc}. 
We would like to thank Professor T.\ Saito  and A.\ Abbes for helpful comments
on this work.
We thank K.\ Nakamura for teaching us the paper \cite{CM}.
We thank Y.\ Mieda for teaching us a conjecture in \cite{W} and pointing out that 
the curve $a^p-a=t^{p+1}$
, which we find in this paper,
is called the Deligne-Lusztig curve for ${\rm SL}_2(\mathbb{F}_p).$
We also thank N.\ Imai for several discussions on this topic and for helpful comments.
We would like to thank T.\ Ito, S.\ Yasuda, T.\ Yamauchi and T.\ Ochiai 
for many suggestions on the applications of our work.
\begin{notation}
We fix some $p$-adic notation.
We let $\mathbb{C}_p$
be the completion of a fixed algebraic closure of $\mathbb{Q}_p,$
with integer ring $\mathbb{R}_p$
and with $\mathfrak{m}_{\mathbb{R}_p}$ 
the maximal ideal of $\mathbb{R}_p.$
For any finite field $\mathbb{F}$
contained in $\bar{\mathbb{F}}:=\mathbb{R}_p/\mathfrak{m}_{\mathbb{R}_p},$
an algebraic closure of $\mathbb{F}_p,$
let $W(\mathbb{F}) \subset \mathbb{R}_p$ denote the ring of Witt
vectors
of $\mathbb{F}.$
Let $v$ denote the unique valuation
on $\mathbb{C}_p$ with $v(p)=1,$
$|\cdot|$
the absolute value given by $|x|=p^{-v(x)}$
and $\mathcal{R}=|\mathbb{C}_p^\ast|=p^{\mathbb{Q}}.$
Throughout the paper, we let $K$
be a complete subfield of $\mathbb{C}_p$
with ring of integers $R_K$ and residue field $\mathbb{F}_K.$
For $r \in \mathcal{R},$
we let $B_K[r]$ and $B_K(r)$
denote the closed and open disk over $K$
of radius $r$
around $0,$ i.e.\
the rigid spaces over $K$
whose $\mathbb{C}_p$-valued points
are $\{x \in \mathbb{C}_p:|x| \leq r\}$ and $\{x \in \mathbb{C}_p:|x|<r\}$
respectively.
If $r, s \in \mathcal{R}$ and $r \leq s,$
let $A_K[r,s]$ and $A_K(r,s)$ be the rigid spaces over $K$
whose $\mathbb{C}_p$-valued points are 
$\{x \in \mathbb{C}_p:r\leq |x| \leq s\}$ and 
$\{x \in \mathbb{C}_p:r<|x|<s\},$
which we call  closed annuli and open annuli.
By the width of such an annulus,
we mean ${\rm log}_p(s/r).$
A closed annuli of width $0$
will be called a circle, which we will also denote the circle, $A_K[s,s],$
by $C_K[s].$

\end{notation}

\section{Preliminaries}
\subsection{Kronecker's polynomial}
We recall Kronecker's polynomial from \cite[Section 0]{dSh}
$$F_p(j,X)=(X-j(p \tau)) \prod_{0 \leq a \leq p-1}(X-j(\frac{\tau+a}{p})).$$
Kronecker proved:
\begin{enumerate}
\item $F_p(j,X) \in \mathbb{Z}[j,X]$ and it gives a plane model
of $X_0(p)$ over $\mathbb{Q}.$
\item $F_p(j,X)=F_p(X,j).$
\item $F_p(j,X)=(X^p-j)(X-j^p)\ ({\rm mod}\ p\mathbb{Z}[j,X]).$
\end{enumerate}
We write $F_p(X,Y)=(X^p-Y)(X-Y^p)+pf(X,Y)$ with some polynomial 
$f(X,Y) \in \mathbb{Z}[X,Y].$

Let $p \geq 13$ be a prime number.
Let $A$ be a supersingular elliptic curve over $\mathbb{F}_p$
with $j(A)\neq 0,1728.$
We denote by $\mathbb{Q}_{p^2}$
the unramified quadratic extension of $\mathbb{Q}_p.$
Let $\mathbb{Z}_{p^2}$ be the integer ring of $\mathbb{Q}_{p^2}.$
Let $\beta \in \mathbb{Z}_{p^2}^\ast$ lift $j(A) \in \mathbb{F}_p.$
We consider the Kronecker polynomial $F_p(X,Y)$ and
 change variables as follows $X_1=X-\beta, Y_1=Y-\beta.$
We put $F^{\beta}_p(X_1,Y_1):=F_p(X_1+\beta,Y_1+\beta).$
\begin{lemma}\label{0}
Let the notation be as above. Then, the followings hold
\begin{enumerate}
\item The polynomial $F^{\beta}_p(X_1,Y_1)$ has the following form
$$(X_1^p-Y_1)(X_1-Y_1^p)+pf_1(X_1,Y_1).$$
\item We have an equality $f_1(X_1,Y_1)=f_1(Y_1,X_1).$
\item $f_1(0,0)$ is a unit.
\end{enumerate}
\end{lemma}
\begin{proof}
We put
$(X+\beta)^p=X^p+\beta^p+pf(X)$ and write
 $pu=\beta^p-\beta$ with some unit $u \in \mathbb{Z}^*_{p^2}.$
By a direct computation, we obtain an equality
$$F_p(X_1+\beta,Y_1+\beta)=(X_1^p-Y_1)(X_1-Y_1^p)+pf_1(X_1,Y_1)$$
with
$f_1(X_1,Y_1)=(X_1-Y_1^p)(u+f(X_1))-(X_1^p-Y_1)(u+f(Y_1))-p(u+f(X_1))(u+f(Y_1))$
$+f(X_1+\beta,Y_1+\beta).$ Hence the assertions 1, 2 follow.
The assertion 3 is a well-known fact.
For example, see \cite[the proof of Corollary 3 in Section 0]{dSh}.
\end{proof}
We write $X,Y$ for $X_1,Y_1.$
Let $c_0$ be the leading coefficient of $f_1.$
We write $f_1(X,Y)=c_0+g(X,Y).$
We write $g(X,Y)=Xg(X)+h(X)Y+g_1(X,Y)Y^2.$
Let $c_1$ be the leading coefficient of $h(X).$ 
The leading coefficient of $g(X)$
is also $c_1$ by Lemma \ref{0}.2.
\begin{lemma}\label{22}
We consider the equation $F^{\beta}_p(X,Y)=0.$
Assume $0<v(X),v(Y)<1.$
\\(a). Assume that $v(X)<v(Y^p).$
Then, $Y$ is written as follows 
\begin{equation}\label{key'2}
Y=X^p+\frac{pc_0}{X}+\frac{pg(X,Y)}{X}+\sum_{n=1}^{\infty}\biggl(\frac{Y^p}{X}\biggr)^nH(X,Y)
\end{equation}
where we put $H(X,Y):=(pc_0/X)+(pg(X,Y)/X).$
\\1. If we assume $0<v(X)<\frac{1}{p+1},$
we have $v(Y)=pv(X).$
\\2. If $v(X)=\frac{1}{p+1},$
we have $v(Y) \geq \frac{p}{p+1}.$
\\3. If $\frac{1}{p+1}<v(X)<\frac{p}{p+1},$
we have $v(X)+v(Y)=1.$
\\(b). Assume $v(X) \geq v(Y^p),$ then the followings hold
\\4. If $v(X)<\frac{p}{p+1},$ we have $v(Y)=v(X)/p.$
\\5. If $v(X) \geq \frac{p}{p+1},$ we have $v(Y)=\frac{1}{p+1}.$
\end{lemma}
\begin{proof}
(a). We assume that $v(X)<v(Y^p).$ First, we prove the equality (\ref{key'2}).
Dividing $F^{\beta}_p(X,Y)=0$ by $X,$ we obtain
$$(X^p-Y)(1-\frac{Y^p}{X})+\frac{pc_0}{X}+\frac{pg(X,Y)}{X}=0.$$
By  the assumption that
$v(X)<v(Y^p),$ $1-(Y^p/X)$ is a unit.
Dividing the above equality
by $\{1-(Y^p/X)\}^{-1},$ we acquire
the equality (\ref{key'2}).
We prove the assertions 1, 2 and 3.
We have an inequality $v(\frac{pg(X,Y)}{X}+\sum_{n \geq 1}\bigl(\frac{Y^p}{X}\bigr)^nH(X,Y))
>v(pc_0/X).$
Comparing the valuation
$v(X^p)$ with the valuation $v(pc_0/X),$
we obtain the assertions 1, 2 and 3 by (\ref{key'2}) immediately. 
We show that the case that $v(X)<v(Y^p), v(X) \geq \frac{p}{p+1}$
can not happen. Now assume that $v(X)<v(Y^p), v(X) \geq \frac{p}{p+1}.$
Considering the equality (\ref{key'2}),
we obtain $v(X)+v(Y)=1$ by $ v(X) \geq \frac{p}{p+1}.$
Therefore we acquire $v(Y) \leq 1/(p+1).$
However, by $v(X)<v(Y^p),$
we obtain an inequality
$v(Y)>1/(p+1).$ Hence, this case can not happen. 

(b). Secondly, we consider the case
$v(X) \geq v(Y^p).$
In particular, we have $v(Y)<v(X^p).$
By the symmetry $F^{\beta}_p(X,Y)=F^{\beta}_p(Y,X)$ in Lemma \ref{0},
switching $X$ and $Y$, we acquire the followings
by the assertions 1,2 and 3 proved above;
\\1'. If we assume $0<v(Y)<\frac{1}{p+1},$
we have $v(X)=pv(Y).$
\\2'. If we assume that $v(Y)=\frac{1}{p+1},$
we have $v(X) \geq \frac{p}{p+1}.$
\\3'. If we assume $\frac{1}{p+1}<v(Y)<\frac{p}{p+1},$
we have $v(X)+v(Y)=1.$
\\To prove 4 and 5,
it suffices to show that the case 3' can not happen.
Now assume $\frac{1}{p+1}<v(Y)<\frac{p}{p+1},$
 $v(X)+v(Y)=1.$
By the assumption $v(X) \geq v(Y^p),$
we have $v(X)>p/(p+1).$
However, by $v(X)+v(Y)=1$ and $\frac{1}{p+1}<v(Y)<\frac{p}{p+1},$
we acquire
$v(X)<p/(p+1).$
Hence it can not happen.
Thus the required assertions follow.
\end{proof}
\begin{corollary}\label{key}
Let the notation be as in Lemma \ref{22}.
\\1. We change variables as follows $X=\alpha x,Y=\alpha^py$
with $x$ a unit and
with an element $\alpha$ satisfying $v(\alpha) \leq \frac{1}{p+1}.$
Then, $y$ is written as follows with respect to $x$
$$y=x^p+\frac{pc_0}{\alpha^{p+1}}\frac{1}{x}+\frac{p}{\alpha^p}\{g(\alpha x)+(\alpha x)^{p-1}h(\alpha x)\}+
\frac{p^2}{\alpha^{p+2}}\frac{h(\alpha x)}{x^2}\{c_0+\alpha x g(\alpha x)\}\
 ({\rm mod}\ p\alpha).$$
\\2.  We change variables as follows $X=\alpha x,Y=\alpha y$
with $x,y$ units and with an element $\alpha$ satisfying $v(\alpha)=1/2.$
Then, the following relationship between $x$ and $y$ holds
$$xy=\frac{pc_0}{\alpha^2}+\frac{pc_1}{\alpha}(x+y)\ 
({\rm mod}\ p).$$
\end{corollary}
\begin{proof}
We prove the assertion 1.
We have the following congruences
$$(1/\alpha^p)\sum_{n=1}^{\infty}(\alpha^{p^2-1}y^p/x)^nH(\alpha x,\alpha^py)=0\ ({\rm mod}\ p\alpha)$$
and $pg(\alpha x,\alpha^py)/\alpha^{p+1}x=
(pg(\alpha x)/\alpha^p)+(ph(\alpha x)y/\alpha x)\ ({\rm mod}\ p\alpha).$
Hence, by substituting $X=\alpha x,Y=\alpha^py$ to (\ref{key'2}) and dividing by $\alpha^p$,
we obtain 
$$y=x^p+\frac{pc_0}{\alpha^{p+1}}\frac{1}{x}+\frac{p}{\alpha^p}g(\alpha x)+\frac{ph(\alpha x)}{\alpha x}y\
 ({\rm mod}\ p\alpha).$$
Therefore, multiplying by $\bigl(1-(p h(\alpha x)/\alpha x)\bigr)^{-1},$
the congruence
$\bigl(1-(p h(\alpha x)/\alpha x)\bigr)^{-1}=1+(p h(\alpha x)/\alpha x)$
$({\rm mod}\ p\alpha)$ gives the required assertion.
\\We prove the assertion 2.
We have the following congruences
$\alpha^{p-1}x^p=0\ ({\rm mod}\ p),$
$$(1/\alpha)\sum_{n=1}^{\infty}(\alpha^{p-1}y^p/x)^nH(\alpha x,\alpha y)=0\ ({\rm mod}\ p)$$ and
$pg(\alpha x,\alpha y)/\alpha^{2}x=(pc_1/\alpha) \times (1+(y/x))\ ({\rm mod}\ p).$
Hence, by substituting $X=\alpha x,Y=\alpha y$ to (\ref{key'2}) and dividing by $\alpha$,
we obtain 
$$y=(pc_0/\alpha^2x)+(pc_1/\alpha) \cdot (1+(y/x))\ ({\rm mod}\ p).$$
Multiplying by $x,$ we obtain the required assertion.
\end{proof}
\subsection{Canonical subgroups and Supersingular Annuli in \cite[Section 3.1]{CM}}

We briefly recall canonical subgroups and supersingular annuli from \cite[Section 3.1]{CM}.
We think of $X_0(p^n)$
as the rigid analytic curve over $\mathbb{Q}_p$
whose points over $\mathbb{C}_p$
are in a one-to-one correspondence
with isomorphism classes of
pairs, $(E,C)$,
where $E/\mathbb{C}_p$
is a generalized elliptic curve
and $C$ is a cyclic subgroup
of order $p^n.$
We implicitly make use of this correspondence
when we speak loosely
of ``the point $(E,C).$''
For integers $n>m \geq 0,$ there are several natural maps from $X_0(p^n)$
to $X_0(p^m)$
which can be defined by way of this moduli-theoretic interpretation of points, and we begin
this subsection by fixing notations for these fundamental maps.
\begin{definition}\label{low}
Let $C[p^i]$
denote the kernel
of multiplication by $p^i$ in $C.$
Let
$$\pi_f,\pi_v:\coprod_{n \geq 1}X_0(p^n)
\longrightarrow \coprod_{n \geq 0}X_0(p^n)$$
be the maps given by $\pi_f(E,C)=(E,pC)$
and $\pi_v(E,C)=(E/C[p],C/C[p])$ respectively.

Let $a,b \in \mathbb{Z}_{ \geq 0}.$
Then by letting
$\pi_{a,b}=\pi^b_f \circ \pi_v^a,$
we obtain maps
$$\pi_{a,b}:\coprod_{n \geq a+b}X_0(p^n)
\longrightarrow \coprod_{n \geq 0}X_0(p^n).$$
\end{definition}
Another map is the Atkin-Lehner involution
$$w:\coprod_{n \geq 0}X_0(p^n)
\longrightarrow \coprod_{n \geq 0}X_0(p^n),$$
which is defined by the formula,
$$w_n(E,C)=(E/C,E[p^n]/C),$$
where $w_n:=w|_{X_0(p^n)}.$
The Atkin-Lehner involution is compatible with the level-lowering 
maps in the sense that $\pi_f \circ w=w \circ \pi_v$
or equivalently, $w \circ \pi_f=\pi_v \circ w,$
since $w$ is an involution.

We introduce some natural subspaces of $X_0(p^n)$
over finite extension of $\mathbb{Q}_p$
using the theory
of the canonical subgroup, which we now review and cite \cite[subsection 3.1]{CM} or
\cite[Section 3]{B} as our references.
If $E$ is an elliptic curve over $\mathbb{C}_p,$
we let $h(E)$ denote the minimum
of $1$
and the valuation
of a lifting of the Hasse invariant of the reduction of a non-singular model of
$E$ mod $p,$ if it exists, and $0$ otherwise.
In \cite[Section 3]{Ka}, Katz constructed a rigid analytic section
$s_1:X_0(p) \longrightarrow X(1)$
over the wide open $W_1$ whose $\mathbb{C}_p$-valued points
are represented by generalized elliptic curves
$E$ such that $h(E)< \frac{p}{p+1},$
when $p \geq 5.$
Both $W_1$ and $s_1$
are defined over $\mathbb{Q}_p.$
We let $K_1(E) \subset E$
denote the subgroup of order $p$ for which $s_1(E)=(E,K_1(E)),$
and we call $K_1(E)$
{\it the canonical subgroup of order $p.$}
\begin{definition}(\cite[Definition 3.1]{B})
Let $L/\mathbb{Q}_p$ be a finite extension,
and say $E/L$ is an elliptic curve.
Say that $E$ is {\it too supersingular}
if $h(E)\geq \frac{p}{p+1}.$
\end{definition}
 We introduce the theorem of Katz-Lubin in \cite[Theorem 3.3]{B} or in \cite[Theorem 3.1]{Ka},
 which studies
 a relationship between the canonical subgroup $K_1(E)$ and the invariant $h(E).$
\begin{theorem}\label{Ka}
(\cite[Theorem 3.3]{B} or \cite[Theorem 3.1]{Ka})
For a not too supersingular elliptic curve $E/L,$
let $K_1(E) \subset E$ be its canonical subgroup.
We have the followings
\begin{enumerate}
\item If $h(E)=0,$
then $K_1(E)$ is the finite \'{e}tale subgroup of $E$
corresponding to the kernel of the reduction
map from $E[p](\overline{L})$
to the $p$-torsion in the N\'{e}ron model
of $E$ over $\mathcal{O}_L.$
\item If $h(E)<\frac{1}{p+1},$
then $E$ is not too supersingular, and $h(E/K_1(E))=ph(E).$
\item If $h(E)=\frac{1}{p+1},$
then $E/K_1(E)$ is too supersingular.
\item If $h(E) \in (\frac{1}{p+1},\frac{p}{p+1}),$
then $E/K_1(E)$ is again not too supersingular,
$h(E/K_1(E))=1-h(E),$ and furthermore the canonical subgroup
 $K_1(E/K_1(E))$ is $E[p]/K_1(E).$
\item If $h(E)<\frac{p}{p+1}$ and $C \neq K_1(E)$
is a subgroup of order $p$, then $h(E/C)=h(E)/p$
and the canonical subgroup $K_1(E/C)$ 
is $E[p]/C.$
\item If $h(E) \geq \frac{p}{p+1}$ and $C \subset E$
is a subgroup of order $p,$
then $h(E/C)=\frac{1}{p+1}$ and the canonical subgroup 
$K_1(E/C)$ is $E[p]/C.$
\end{enumerate}
\end{theorem}
\begin{remark}
With Corollary \ref{key3} and Remark \ref{rem.}, by Lemma \ref{22}, we can 
recover a part of the statements in Theorem
\ref{Ka}.
\end{remark}
For $n \geq 1,$
we generalize $W_1$ by taking $W_n$
to be the wide open in $X(1)$
where $h(E)>\frac{p^{2-n}}{p+1}.$
For $E \in W_n,$
we define $K_n(E)$
inductively, as in \cite[Definition 3.4]{B},
as the preimage of $K_{n-1}(E/K_1(E))$
under the natural projection $E \longrightarrow E/K_1(E).$
This is a cyclic subgroup of order $p^n$ when $E \in W_n$
 and we call it {\it the canonical subgroup of order $p^n.$}
 
Thus, when $E$ has supersingular reduction, either $h(E) \geq \frac{p}{p+1},$
or there is a largest $n \geq 1$
for which $K_n(E)$ can be defined.
In the first case, we define {\it the canonical subgroup of E},
denoted by $K(E),$ to be the trivial subgroup,
and in the second case,
we let $K(E)=K_n(E)$ for this largest $n.$
\begin{definition}(\cite[Definition 3.3]{CM})
For a fixed elliptic curve $A$ over a finite field $\mathbb{F},$
let $W_A(p^n)$ represent the rigid subspace of $X_0(p^n)$
(over $\mathbb{Q}_p \otimes W(\mathbb{F})$) whose points over $\mathbb{C}_p$
are represented by pairs $(E,C)$ with $\overline{E} \simeq A.$ 
\end{definition}
Of cource, $W_A(1)$ (for any $A$)
is just a residue disk of the $j$-line.
When $A$ is a supersingular elliptic curve, it is well-known
that $W_A(p)$ is isomorphic,
over $\mathbb{Q}_{p^2},$
to an open annulus of width $i(A):=|{\rm Aut}(A)|/2.$
This means that one can choose
a parameter $x_A$ on $W_A(p)$
over $\mathbb{Q}_{p^2}$
which identifies it with
the open annulus $A_{\mathbb{Q}_{p^2}}(p^{-i(A)},1).$
In fact, we can and will always do this in such a way that $v(x_A(E,C))=i(A)h(E)$
when $C=K_1(E)$ and otherwise
$i(A)(1-h(E/C)).$

In \cite[Section 3.1]{CM}, Coleman-McMurdy consider the following concentric circles
under the above identification $W_A(p) \simeq A_{\mathbb{Q}_{p^2}}(p^{-i(A)},1)$
\begin{enumerate}
\item $\mathbf{SD}_A:=C[p^{-\frac{i(A)}{2}}]$
which they call the ``self-dual circle'' or the ``Atkin-Lehner circle,''
\item $\mathbf{TS}_A:=C[p^{-\frac{pi(A)}{p+1}}]$
which they call the ``too-supersingular circle.''
\end{enumerate}
Then, these spaces have the following descriptions
by using Theorem \ref{Ka}
$$\mathbf{SD}_A=\{(E,C) \in W_A(p)|h(E)=1/2,C=K_1(E)\},$$
$$\mathbf{TS}_A=\{(E,C) \in W_A(p)|E:{\rm too}\ {\rm supersingular},h(E/C)=1/(p+1)\}.$$
Generally, for a rational number $q \in (0,1),$
we denote a circle $C[p^{-qi(A)}]$ by $\mathbf{C}^A_q.$
We fix a point $(E,C) \in \mathbf{C}^A_q.$
If we assume that $q < p/(p+1),$
the subgroup $C$ is always the canonical subgroup
again by using Theorem \ref{Ka}.
From now on, we assume that $j(A) \neq 0,1728.$
Hence, we have $i(A)=1.$
Coleman-McMurdy study the following
subspaces $\mathbf{Y}^A_{1,1}:=\pi_v^{-1}(\mathbf{TS}_A) \subset W_A(p^2)$
in \cite[Section 5]{CM},
and $\mathbf{Z}^A_{1,1}:=\pi^{-1}_{1,1}(\mathbf{SD}_A) \subset W_A(p^3)$
in \cite[Section 8]{CM}. 
They calculate the reductions of
$\mathbf{Y}^A_{1,1}$ and $\mathbf{Z}^A_{1,1},$
and deduce the stable reductions
of $X_0(p^2)$ and $X_0(p^3).$
The reduction $\overline{\mathbf{Y}}^A_{1,1}$ which appears in the stable reduction
of $X_0(p^2)$
corresponds to the horizontal component in \cite[Theorem 2.1.2]{E}.
The main part in \cite{CM} is in computing the reduction of the space
$\mathbf{Z}^A_{1,1}.$
We have the following descriptions
$$\mathbf{Y}^A_{1,1}=\{(E,C) \in W_A(p^2)|E/pC:{\rm too}\ {\rm supersingular}, K(E)=pC\},$$
$$\mathbf{Z}^A_{1,1}=\{(E,C) \in W_A(p^3)|h(E)=1/2p,|K(E) \cap C|=p^2\}$$
by using Theorem \ref{Ka}.
Roughly speaking, to calculate the reductions of $\mathbf{Y}^A_{1,1}$ and $\mathbf{Z}^A_{1,1}$
in \cite[Lemma 5.1 and Proposition 7.1]{CM},
Coleman-McMurdy consider embeddings of these spaces 
into products of subspaces of $W_A(p)$ and apply de Shalit's 
approximation theorem for $\pi_f$ in \cite[Theorem 3.5]{CM}.

We recall de Shalit's approximation theorem for $\pi_f$ in
\cite[Theorem 3.5]{CM}.
\begin{theorem}\label{thm}
(\cite[Theorem 3.5]{CM})
Let $A$ be a supersingular elliptic curve with
$j(A) \neq 0,1728.$ 
There are parameters $s$ and $y$ over $\mathbb{Z}_{p^2}$ which identify $W_A(1)$
with the disk $B_{\mathbb{Q}_{p^2}}(1)$ and $W_A(p)$ with the annulus
$A_{\mathbb{Q}_{p^2}}(p^{-1},1),$
and series $F(T),G(T) \in T\mathbb{Z}_{p^2}[[T]],$
such that 
\\1. $w_1^\ast(y)=\kappa/y$ for some $\kappa \in \mathbb{Z}_{p^2}$ with $v(\kappa)=1.$
\\2. $\pi_f^\ast s=F(y)+G(\kappa/y),$ where
 \\(a). $F'(0)=1\ ({\rm mod}\ p),$ and
 \\(b). $G(T)=(F(T))^p\ ({\rm mod}\ p).$
\end{theorem}
As mentioned in \cite[the proof of Theorem 3.5]{CM},
our $\kappa$ and $\pi_f^\ast s$ are equal to $\pi$ and $\psi(y)-\beta_0$
respectively in de Shalit's language in \cite[Section 3]{dSh}.
We have $\beta_0 \in \mathbb{Z}_{p^2}^{\ast}$
and it is a lifting of $j(A) \in \mathbb{F}_p$
by loc.\ cit.
\begin{lemma}\label{key2}
Let the notation be as in Theorem \ref{thm}.
Then, the space $W_A(p)$ is embedded into a product 
$(s_1,s_2) \in W_A(1) \times W_A(1)$
by a map $(\pi_f,\pi_v)$ and
its image is defined by 
the following equation
$F_p^{\beta_0}(s_1,s_2)=0.$ 
\end{lemma}
\begin{proof}
We consider the embedding
$$(\pi_f,\pi_v):W_A(p) \hookrightarrow W_A(1) \times W_A(1).$$
Since we have $\pi_v=\pi_f \circ w_1,$ we acquire 
$\pi_f^\ast (s_1)=\psi(y)-\beta_0$
and $\pi_v^\ast (s_2)=\psi(\kappa/y)-\beta_0$ by Theorem \ref{thm}.1.
Thereby, we obtain $F_p^{\beta_0}(\pi_f^\ast (s_1),\pi_v^\ast (s_2))=
F_p(\psi(y),\psi(\kappa/y))=0$ by \cite[\textbf{3.3}]{dSh}.
Hence, the assertion follows.
\end{proof}
For a rational number $q \in (0,1),$
let $\mathbf{C}^{A,0}_q$ denote a circle
$C[p^{-i(A)q}] \subset W_A(1)$
via the parameter $s$ chosen as in Theorem \ref{thm}.
Similarly,
let $\mathbf{C}^{A,0}_{\geq q}$
denote a closed disk
$B_{\mathbb{Q}_{p^2}}[p^{-i(A)q}] \subset W_A(1).$
\begin{corollary}\label{key3}
Let the notation be as in Lemma \ref{key2}.
Let $q \in (0,1)$ be a rational number.
We have the followings
\\1. If we assume $q < 1/(p+1),$
the circle $\mathbf{C}^{A}_{q}$
is embedded into the following product 
$\mathbf{C}^{A,0}_{q} \times \mathbf{C}^{A,0}_{pq}$
by the map $(\pi_f,\pi_v)$ and its image is determined by $F_p^{\beta_0}(s_1,s_2)=0$
where $(s_1,s_2)$ denotes the coordinate of $\mathbf{C}^{A,0}_{q} \times \mathbf{C}^{A,0}_{pq}.$
\\2. If we assume $q=1/(p+1),$
the circle $\mathbf{C}^{A}_{1/(p+1)}$
is embedded into the following product 
$\mathbf{C}^{A,0}_{1/(p+1)} \times \mathbf{C}^{A,0}_{\geq p/(p+1)}$
by the map $(\pi_f,\pi_v)$ and its image is determined by $F_p^{\beta_0}(s_1,s_2)=0.$
\\3. If we assume $1/(p+1) < q < p/(p+1),$
the circle $\mathbf{C}^{A}_{q}$
is embedded into the following product 
$\mathbf{C}^{A,0}_{q} \times \mathbf{C}^{A,0}_{1-q}$
by the map $(\pi_f,\pi_v)$ and its image is determined by $F_p^{\beta_0}(s_1,s_2)=0.$
\\4. If we assume $q=p/(p+1),$
the circle $\mathbf{C}^{A}_{p/(p+1)}$
is embedded into the following product 
$\mathbf{C}^{A,0}_{\geq p/(p+1)} \times \mathbf{C}^{A,0}_{1/(p+1)}$
by the map $(\pi_f,\pi_v)$ and its image is determined by $F_p^{\beta_0}(s_1,s_2)=0.$
\\5. If we assume $q > p/(p+1),$
the circle $\mathbf{C}^{A}_{q}$
is embedded into the following product 
$\mathbf{C}^{A,0}_{p(1-q)} \times \mathbf{C}^{A,0}_{1-q}$
by the map $(\pi_f,\pi_v)$ and its image is determined by $F_p^{\beta_0}(s_1,s_2)=0.$
\end{corollary}
\begin{proof}
We prove 1.
Since we have
$v(y)=q < 1/(p+1),$
we obtain
$v(F(y)+G(\kappa/y))=v(y)=q$ and
$v(F(\kappa/y)+G(y))=v(y^p)=pq$
by Theorem \ref{thm}.2.
Hence the assertion follows from
$\pi_f^\ast (s_1)
=F(y)+G(\kappa/y)$ and $\pi_v^\ast (s_2)=F(\kappa/y)+G(y)$
by Theorem \ref{thm}.
Other required assertions follow in a similar way.
We omit their proofs. 
\end{proof}
\begin{remark}\label{rem.}
Let the notation be as in Corollary \ref{key3}.
We fix a point $y=(E,C) \in \mathbf{C}^{A}_{q}$ for an arbitrary rational number $q \in (0,1).$
We consider the embedding
$(\pi_f,\pi_v):\mathbf{C}^{A}_{q} \hookrightarrow W_A(1) \times W_A(1)$
as in Corollary \ref{key3}. We denote by $(s_1,s_2)$
the coordinate of $W_A(1) \times W_A(1).$
Then, the valuations $v(s_1)$ and $v(s_2)$ are equal to 
the Hasse invariants
$h(E)$ and $h(E/C)$ respectively.
\end{remark}

\subsection{Subspaces of $X_0(p^n)$}
Let $A/\mathbb{F}_p$ be a supersingular elliptic curve.
We define subspaces $\mathbf{Y}^A_{a,b},\mathbf{Z}^A_{a,b}$ 
of $W_A(p^n)$ on the basis of ideas in \cite[Section 5]{Mc}, 
whose reductions are expected to play a key role in the stable model of the modular curve $X_0(p^n).$
In fact, the reductions of the spaces $\mathbf{Y}^A_{a,b},\mathbf{Z}^A_{a,b}$
 appear in the stable reductions of $X_0(p^n)$
for $n=2,3,4.$
In Corollary \ref{prod}, we introduce some embeddings
of the spaces $\mathbf{Y}^A_{a,b},\mathbf{Z}^A_{a,b}$
into products of $W_A(1)$ and determine its images.
Using these identifications,
we will calculate the reductions of the spaces later in cases $n=2,3,4.$

Let the notation be as in the previous subsection.
Let $n$ be a positive integer.
Let $a,b \geq 1$ be positive integers.
We define
$$\mathbf{Y}^A_{a,b}:=\pi^{-1}_{a,b-1}(\mathbf{TS}_A) \subset W_A(p^n)$$
with $a+b=n, n \geq 2$ and 
$$\mathbf{Z}^A_{a,b}:=\pi_{a,b}^{-1}(\mathbf{SD}_A) \subset W_A(p^n)$$
with $a+b=n-1, n \geq 3.$
In lemma below, we give moduli-theoretic descriptions of these spaces.
\begin{lemma}\label{le}
Let the notation be above.
\\1. We have the following equality
$$\mathbf{Y}^A_{a,b}=\{(E,C) \in W_A(p^n)|h(E)=\frac{1}{p^{a-1}(p+1)}, E/p^bC:{\rm too}\  {\rm supersingular},
 K_a(E)=p^bC\}$$ with $a+b=n \geq 2.$ 
For a point $(E,C) \in \mathbf{Y}^A_{a,b},$ we have the followings
\\(a). For $1 \leq i \leq a,$ the subgroup $p^{b+i-1}C/p^{b+i}C \subset E/p^{b+i}C$
is the canonical subgroup and we have $h(E/p^{b+i}C)=1/p^{i-1}(p+1).$
\\(b). For $0 \leq i \leq b-1,$
the subgroup $p^{b-i-1}C/p^{b-i}C \subset E/p^{b-i}C$
is non-canonical.
We have $h(E/p^bC) \geq p/(p+1).$
For $1 \leq i \leq b,$ we have $h(E/p^{b-i}C)=1/p^{i-1}(p+1).$
\\2. We have the following equality
$$\mathbf{Z}^A_{a,b}=
\{(E,C) \in W_A(p^n)|h(E)=1/2p^a, |K(E) \cap C|=p^{a+1}\}$$
with $a+b=n-1 \geq 2.$
For a point $(E,C) \in \mathbf{Z}^A_{a,b},$ we have the followings
\\(a). For $1 \leq i \leq a+1,$
the subgroup $p^{b+i-1}C/p^{b+i}C \subset E/p^{b+i}C$
is the canonical subgroup and we have 
$h(E/p^{b+i}C)=1/2p^{i-1}.$
\\(b). For $0 \leq i \leq b-1,$ the subgroup
$p^{b-i-1}C/p^{b-i}C \subset E/p^{b-i}C$
is non-canonical.
For $0 \leq i \leq b,$
we have $h(E/p^{b-i}C)=1/2p^i.$
\end{lemma}
\begin{proof}
We prove the assertion $1$. We fix a point $(E,C) \in \mathbf{Y}^A_{a,b}.$
The first equality follows from (a),(b).
By the definition of $\mathbf{Y}^A_{a,b},$
$E/p^bC$ is a too supersingular elliptic curve,
hence
$p^{b-1}C/p^bC \subset E/p^bC$
is non-canonical and $h(E/p^bC) \geq p/(p+1).$
If we assume that $p^bC/p^{b+1}C \subset E/p^{b+1}C$
is non-canonical, by Theorem \ref{Ka}.5 and 6,
we have $h(E/p^bC) \leq 1/(p+1).$
This is a contradiction.
Thereby $p^bC/p^{b+1}C \subset E/p^{b+1}C$ is the canonical subgroup
and by Theorem \ref{Ka}.2, 3 and 4,
we have $h(E/p^{b+1}C)=1/(p+1).$
By repeating similar arguments using Theorem
\ref{Ka},
we acquire (a).
The assertion (b) is shown in the same way.
The assertion 2 is also shown in a very similar manner as above.
We omit the proof.
\end{proof}

We give analogous 
descriptions of $\mathbf{Y}^A_{1,1}$ and $\mathbf{Z}^A_{1,1}$
to the ones \cite[Lemma 5.1 and Proposition 7.1]{CM}.

\begin{lemma}\label{emb}
1. The restriction of the map 
$\pi_f \times \pi_v:W_A(p^2) \hookrightarrow W_A(p)^{\times 2}$
 to the subspace $\mathbf{Y}^A_{1,1}$ induces the following isomorphism
$$\iota:\mathbf{Y}^A_{1,1} \overset{\sim}{\longrightarrow} 
\{(x,y) \in \mathbf{C}^A_{\frac{1}{p+1}} \times \mathbf{C}^A_{\frac{p}{p+1}}|
\pi_v(x)=\pi_f(y), w_1(x) \neq y\}=:S_1$$
\\2. The restriction of the map 
$\pi_{2,0} \times \pi_{1,1} \times \pi_{0,2}:W_A(p^3) \hookrightarrow W_A(p)^{\times 3}$
 to the subspace $\mathbf{Z}^A_{1,1}$ induces the following isomorphism
$$i:\mathbf{Z}^A_{1,1} \overset{\sim}{\longrightarrow}
\{(x,y,z) \in \mathbf{C}^A_{\frac{1}{2p}} \times \mathbf{C}^A_{\frac{1}{2}}
\times \mathbf{C}^A_{1-\frac{1}{2p}}|\pi_v(x)=\pi_f(y),\pi_v(y)=\pi_f(z)\}=:T.$$
\end{lemma}
\begin{proof}First, we prove the assertion $1.$
The well-definedness of the map $\iota$ follows from Lemma \ref{le}.
For an element $(x,y)=((E_1,C_1),(E_2,C_2)) \in S_1,$
the subgroup $C_1 \subset E_1$ is the canonical subgroup, and $C_2 \subset E_2$
is not the canonical subgroup by Theorem \ref{Ka}.
Further, we have $h(E_1)=h(E_2/C_2)=1/(p+1)$ and $E_2$ is a too supersingular elliptic curve
by Theorem \ref{Ka}.
It suffices to show the surjectivity of the map $\iota$. 
We fix an element $(x,y)=((E_1,C_1),(E_2,C_2)) \in S_1.$
Then, we consider the composite
$$\psi:E_1 \longrightarrow E_1/C_1 \simeq E_2 \longrightarrow E_2/C_2.$$
We put $C:={\rm Ker}\ \psi.$ We simply write $E$ for $E_1.$
As mentioned above, we have $C_1=K_1(E).$
Then, we acquire the following sequence
$$E \longrightarrow E/K_1(E) \simeq E_2 \longrightarrow E/C \simeq E_2/C_2.$$
Since we have $p(C/K_1(E))=0$ in $E/K_1(E)$, we acquire an inclusion $pC \subset K_1(E) \subset C.$
By the inclusion, either $pC=0$ or $pC=K_1(E)$ can happen.
In the following, we show that the case $pC=0$ does not occur.
Assuming $pC=0,$ we deduce a contradiction.
We have $C \subset E[p]$ and hence $C=E[p],$ since $C$ is a subgroup of order $p^2.$
Therefore, we obtain $y=(E_2,C_2) \simeq (E/K_1,E[p]/K_1)=w_1(x).$
However, this is a contradiction, by the assumption $w_1(x) \neq y.$
Hence, we obtain $pC=K_1(E).$ By $h(E)=1/(p+1),$
we acquire that $E/pC=E/K_1$ is a too supersingular elliptic curve by Theorem \ref{Ka}.3. 
Thereby, we obtain $(E,C) \in \mathbf{Y}^A_{1,1}$
and $\iota(E,C)=(x,y).$ Hence, the assertion follows.

We prove the assertion $2.$ The well-definedness of $i$
 follows from Lemma \ref{le}
as in the proof of $1.$
It suffices to show the surjectivity of the map $i.$
Let $(x,y,z)=((E_1,C_1),(E_2,C_2),(E_3,C_3)) \in T.$
We simply write $E$ for $E_1.$
Note that $C_1=K_1(E), C_2=K_1(E_2)$ and the subgroup $C_3 \subset E_3$
is non-canonical. By Theorem \ref{Ka}, we have $h(E)=1/2p,h(E/K_1(E))=h(E/K_2(E))=1/2.$
We consider the following sequence
$$\beta:E \longrightarrow E/C_1 \simeq E_2 \longrightarrow E_2/C_2 \simeq E_3 \longrightarrow E_3/C_3$$
and put $C:={\rm Ker}\ \beta.$
Since the subgroup $C_3 \simeq C/K_2(E) \subset E_3 \simeq E/K_2(E)$ 
is a cyclic subgroup of order $p,$
we acquire $pC \subset K_2(E) \subset C.$
Hence, we obtain $p^2C \subset K_1(E).$
Thereby, either $p^2C=0$ or $p^2C=K_1(E)$ can happen.

In the following, we show that the case $p^2C=0$ can not happen.
Assuming  $p^2C=0,$ we deduce a contradiction.
We have $K_1(E) \subset pC \subset E[p].$
Either $pC=K_1(E)$ or $pC=E[p]$ can occur.
First, we deduce a contradiction, assuming $pC=E[p].$
Since we have $pC \subset K_2(E)$ and $pC$ is a subgroup of order $p^2,$
we acquire $pC=K_2(E)$ and hence $E[p]=K_2(E).$
But, this is a contradiction.
Second, we deduce a contradiction, assuming $pC=K_1(E).$
Let $\alpha:(E/K_1(E))/(K_2(E)/K_1(E)) \overset{\sim}{\longrightarrow} E/K_2(E)$
be the canonical isomorphism.
Since we have $h(E/K_1(E))=1/2,$
by Theorem \ref{Ka}, 
$(E/K_1(E))[p]/(K_2(E)/K_1(E)) \subset (E/K_1(E))/(K_2(E)/K_1(E))$
is the canonical subgroup.
By the assumption $pC=K_1(E),$
we have $\alpha^{-1}(C/K_2(E))=(E/K_1(E))[p]/(K_2(E)/K_1(E)).$
Since the subgroup $C/K_2(E)$ is non-canonical in $E/K_2(E)$, this is a contradiction.
Hence, we have proved that the case $p^2C=0$ can not happen.
Therefore, we obtain $p^2C=K_1(E).$ By the inclusion 
$p^2C=K_1(E) \subset pC \subset K_2(E),$
we acquire $pC=K_2(E).$
Thereby,  $(E,C) \in \mathbf{Z}^A_{1,1}$ and $i(E,C)=(x,y,z)$ hold.
Hence, the assertion follows.
\end{proof}

\begin{remark} Let the notation be as in Lemma \ref{emb}.
\\1. We clarify a relationship between our embeddings of 
$\mathbf{Y}^A_{1,1}, \mathbf{Z}^A_{1,1}$ in Lemma \ref{emb}
and the ones which were
considered in \cite[Lemma 5.1 and Proposition 7.1]{CM}.
Coleman-McMurdy considered the following isomorphism in \cite[Lemma 5.1]{CM}
$$\iota':\mathbf{Y}^A_{1,1} \overset{\sim}{\longrightarrow}
\{(x,y) \in \mathbf{TS}_A \times \mathbf{TS}_A|\pi_f(x)=\pi_f(y), x \neq y\}=:S_2$$
which sends $(E,C)$ to
$((E/pC,E[p]/pC),(E/pC,C/pC)).$
The Atkin-Lehner involution induces 
a well-defined map $w_1 \times 1:S_1 \longrightarrow S_2.$
Then, we can easily check that the following diagram is commutative
\[\xymatrix{
\mathbf{Y}^A_{1,1} \ar[r]^{\iota}_{\simeq}\ar[dr]^{\iota'}_{\simeq} & S_1 \ar[d]^{w_1 \times 1}_{\simeq} \\
 & S_2.  \\}
\]
\\2. We put $\mathbf{C}_A:=\mathbf{C}^A_{1-\frac{1}{2p}}$ as in \cite[Remark 3.4]{CM},
which they call the anti-Atkin-Lehner involution.
Then, we have a rigid analytic map $\tau_f:\mathbf{C}_A \longrightarrow \mathbf{SD}_A=\mathbf{C}^A_{\frac{1}{2}}$
which sends $(E,C')$ to $(E,K_1(E))$ as in \cite[Subsection 3.1]{CM}.
Coleman-McMurdy defines an isomorphism in \cite[Proposition 7.1]{CM}
$$i':\mathbf{Z}^A_{1,1} \overset{\sim}{\longrightarrow}
\{(x,y) \in \mathbf{C}_A \times \mathbf{C}_A|w_1 \circ \tau_f(x)=\tau_f(y)\}=:S$$
which sends $(E,C)$ to
$((E/p^2C,E[p]/p^2C),(E/pC,C/pC)).$
We define an isomorphism $\phi:S \longrightarrow
T$ by 
$$(x,y) \mapsto (w_1(x),\tau_f(x),y).$$
Then, the inverse map is described as follows
$$\psi:T \longrightarrow S; ((E_1,C_1),(E_2,C_2),(E_3,C_3)) \mapsto 
((E_2,D),(E_3,C_3))$$
where $D \subset E_2$ is a cyclic subgroup of order $p$
defined by the following diagram
\[\xymatrix{
E_1/C_1 \ar[r]^{\simeq} & E_2  \\
E_1[p]/C_1 \ar[u]^{\cup}\ar[r]^{\simeq} & D \ar[u]^{\cup}. \\}
\]
Then, we can easily check that the following diagram is commutative
\[\xymatrix{
\mathbf{Z}^A_{1,1} \ar[r]^{i}_{\simeq}\ar[dr]^{i'}_{\simeq} & T \ar[d]^{\psi}_{\simeq} \\
 & S.  \\}
\]
\end{remark}

\begin{lemma}\label{emb2}
Let $a, b$ be positive integers.
We have the followings
\\1. For $a \geq 2, a+b=n,$ the restriction of 
the map $\pi_{0,n-1} \times \pi_v:W_A(p^n) \hookrightarrow W_A(p) \times W_A(p^{n-1})$ 
 to the subspace $\mathbf{Y}^A_{a,b}$ induces the following isomorphism 
\begin{equation}\label{tt1}
\mathbf{Y}^A_{a,b}
\simeq \{(x,y) \in \mathbf{C}^A_{1/p^{a-1}(p+1)} \times \mathbf{Y}^A_{a-1,b}|\pi_v(x)=\pi_{0,n-1}(y)\}.
\end{equation}
\\2. For $b \geq 2, a+b=n,$ the restriction of
the map $\pi_f \times \pi_{n-1,0}:W_A(p^n) \hookrightarrow W_A(p^{n-1}) \times W_A(p)$
to the subspace $\mathbf{Y}^A_{a,b}$
induces the following isomorphism
$$\mathbf{Y}^A_{a,b}
\simeq \{(x,y) \in \mathbf{Y}^A_{a,b-1} \times \mathbf{C}^A_{1-(1/p^{b-1}(p+1))}|\pi_f(y)=\pi_{n-1,0}(x)\}.$$
\\3. For $a \geq 2, a+b=n-1,$ the restriction of 
the map $\pi_{0,n-1} \times \pi_v:W_A(p^n) \hookrightarrow W_A(p) \times W_A(p^{n-1})$
 to the subspace $\mathbf{Z}^A_{a,b}$ induces the following isomorphism
$$\mathbf{Z}^A_{a,b}
\simeq \{(x,y) \in \mathbf{C}^A_{1/2p^{a}} \times \mathbf{Z}^A_{a-1,b}|\pi_v(x)=\pi_{0,n-1}(y)\}.$$  
\\4. For $b \geq 2, a+b=n-1,$ the restriction of the map $\pi_f \times \pi_{n-1,0}:W_A(p^n) \hookrightarrow 
W_A(p^{n-1}) \times W_A(p)$ to the subspace $\mathbf{Z}^A_{a,b}$
induces the following isomorphism
$$\mathbf{Z}^A_{a,b}
\simeq \{(x,y) \in \mathbf{Z}^A_{a,b-1} \times \mathbf{C}^A_{1-(1/2p^{b})}|\pi_f(y)=\pi_{n-1,0}(x)\}.$$
\end{lemma}
\begin{proof}
We prove the assertion $1.$ The well-definedness of the map 
$(\pi_{0,n-1} \times \pi_{v})|_{\mathbf{Y}^A_{a,b}}$
follows from Lemma \ref{le}.
The injectivity of the map is trivial.
To prove $1,$ it suffices to show that the map is surjective.
Let $(x,y)=((E_1,C_1),(E_2,C_2))$ be an element of the right hand side of (\ref{tt1}).
By Theorem \ref{Ka} and Lemma \ref{le}, we know that $C_1 \subset E_1$ is the canonical subgroup of order $p.$
We simply write $E$ for $E_1.$
We consider the composite
$$\alpha:E \longrightarrow E/K_1(E) \simeq E_2 \longrightarrow E_2/C_2$$
and put $C:={\rm Ker}\ \alpha.$
Then, the subgroup $C/K_1(E) \subset E/K_1(E)$ is a cyclic subgroup of order $p^{n-1}.$
Hence, we obtain $p^{n-1}C \subset K_1(E).$
Since we have $(E_2,C_2) \simeq (E/K_1(E),C/K_1(E)) \in \mathbf{Y}^A_{a-1,b},$
we acquire $p^b(C/K_1(E))=K_a(E)/K_1(E).$ Thereby, we obtain $p^bC \subset K_a(E).$
By the inclusion $p^{n-1}C \subset K_1(E),$
either $p^{n-1}C=0$ or $p^{n-1}C=K_1(E)$ can happen.

In the following, we show that the case $p^{n-1}C=0$ can not happen.
Assuming $p^{n-1}C=0$, we deduce a contradiction.
The condition $p^{n-1}C=0$ implies $p^{n-2}C \subset E[p].$
By $p^{n-2}(C/K_1(E)) \neq 0$ in $E/K_1(E),$
we acquire $p^{n-2}C \neq 0$ and $p^{n-2}C \neq K_1(E).$
Hence, the subgroup $p^{n-2}C$ is a non-canonical subgroup of order $p$ or 
$p^{n-2}C$ is equal to $E[p].$ 

We show the second case $p^{n-2}C=E[p]$ does not happen.
By multiplying the inclusion $p^bC \subset K_a(E)$ by $p^{a-2},$
we acquire $p^{n-2}C \subset K_2(E).$
Since $p^{n-2}C$ is a subgroup of order $p^2,$ the equality $p^{n-2}C=K_2(E)$
must hold.
Therefore, we acquire $E[p]=K_2(E).$ But this is a contradiction.
Secondly, assuming that the subgroup $p^{n-2}C$ is a non-canonical subgroup of order $p,$ 
we deduce a contradiction.
By $p^{n-2}C \subset K_2(E),$ we have the following commutative diagram
\[\xymatrix{
E \ar[r]^{\!\!\!\!\!\!{\rm deg.} p^2}\ar[dr]_{{\rm deg.} p} & E/K_2(E)  \\
 & E/p^{n-2}C \ar[u]_{{\rm deg.} p}.  \\}
\]
By Theorem \ref{Ka}, we have $h(E/K_2(E))=p^2h(E)$ and $h(E/p^{n-2}C)=h(E)/p.$
The isogeny $E/p^{n-2}C \longrightarrow E/K_2(E)$ of degree $p$ does not exist again by Theorem \ref{Ka}.
Hence, this is a contradiction.

Now, we consider the case $p^{n-1}C=K_1(E).$ 
In this case, we can easily check that $(E,C)$ is in $\mathbf{Y}^A_{a,b}$
and $(E,C)$ goes to $(x,y)$ by the map 
$\pi_{0,n-1} \times \pi_v.$ We have proved the surjectivity of the map.
Hence, the assertion $1$ follows.
Other assertions 2,3,4 follow in the same way as the proof of $1$ above.
We omit the proofs.
\end{proof}

\begin{corollary}\label{emb3}
Let $a,b$ be positive integers.
\\1. The space $\mathbf{Y}^A_{a,b}$ with $a+b=n \geq 2$ 
is embedded into the following product of the subspaces of $W_A(p)$
by a map
$\prod_{0 \leq i \leq n-1}\pi_{i,n-i}$
$$(\{Y_i\}_{0 \leq i \leq n-1}) \in \prod_{0 \leq i \leq a-1}\mathbf{C}^{A}_{1/p^{a-i-1}(p+1)} \times 
\prod_{a \leq i \leq n-1}\mathbf{C}^{A}_{1-(1/p^{i-a}(p+1))}$$
and its image is determined by the following equations
$\{\pi_v(Y_i)=\pi_f(Y_{i+1})\}_{0 \leq i \leq n-2}$
and $w_1(Y_{a-1}) \neq Y_{a}.$
\\2. The space $\mathbf{Z}^A_{a,b}$ with $a+b=n-1 \geq 2$ is embedded into the following 
product of the subspaces of $W_A(p)$ by a map
$\prod_{0 \leq i \leq n-1}\pi_{i,n-i}$
$$(\{Y_i\}_{0 \leq i \leq n-1}) \in
\prod_{0 \leq i \leq a}\mathbf{C}^A_{1/2p^{a-i}} \times 
\prod_{a+1 \leq i \leq n-1}
\mathbf{C}^A_{1-(1/2p^{i-a})}$$
and its image is determined by the following equations
$\{\pi_v(Y_i)=\pi_f(Y_{i+1})\}_{0 \leq i \leq n-2}.$
\end{corollary}
\begin{proof}
We consider an order $(a,b) \geq (a',b') 
\Longleftrightarrow a \geq a', b \geq b'$ on
$\mathbb{Z}_{\geq 1}^{\times 2}.$
We prove the assertion $1$ by an induction on $(a,b) \in \mathbb{Z}_{\geq 1}^{\times 2}$
with respect to this order.
For $(a,b)=(1,1),$ the assertion follows from Lemma \ref{emb}.
Assuming the assertion $1$ for the case $(a,b-1)$ or $(a-1,b),$ 
we prove the assertion for $(a,b).$
By the induction hypothesis and Lemma \ref{emb2}, the assertion follows immediately.
We prove the assertion $2$ in the same way as $1.$
We omit the proof.
\end{proof}

In the following corollary, we consider embeddings of the spaces
$\mathbf{Y}^A_{a,b} \subset W_A(p^n)$ with $a+b=n \geq 2$
and $\mathbf{Z}^A_{a,b} \subset W_A(p^n)$ with $a+b=n-1 \geq 2$
into products of subspaces of $W_A(1),$ and determine its images.
These identifications play an important role in
computations of the reductions of the spaces $\mathbf{Y}^A_{a,b}\ (a+b \leq 4)$
 and $\mathbf{Z}^A_{a,b}\ (a+b \leq 3).$

\begin{corollary}\label{prod}
We have the followings
\\1. The space $\mathbf{Y}^A_{a,b}$ with $a+b=n \geq 2$ 
is embedded into the following product of the subspaces of $W_A(1)$
by a map
$\prod_{0 \leq i \leq n}\pi_{i,n-i}:W_A(p^n) \hookrightarrow W_A(1)^{\times (n+1)}$
$$(\{X_i\}_{0 \leq i \leq n}) \in \prod_{0 \leq i \leq a-1}\mathbf{C}^{A,0}_{1/p^{a-i-1}(p+1)} \times 
\mathbf{C}^{A,0}_{\geq p/(p+1)} \times
\prod_{a+1 \leq i \leq n}\mathbf{C}^{A,0}_{1/p^{i-a-1}(p+1)}$$
and its image is determined by the following equations
$\{F_p^{\beta_0}(X_i,X_{i+1})=0\}_{0 \leq i \leq n}$
and $X_{a-1} \neq X_{a+1}.$
\\2. The space $\mathbf{Z}^A_{a,b}$ with $a+b=n-1 \geq 2$
is embedded into the following product of the subspaces of $W_A(1)$
by a map
$\prod_{0 \leq i \leq n}\pi_{i,n-i}$
$$(\{X_i\}_{0 \leq i \leq n}) \in \prod_{0 \leq i \leq a}\mathbf{C}^{A,0}_{1/2p^{a-i}} \times 
\prod_{a+1 \leq i \leq n}
\mathbf{C}^{A,0}_{1/2p^{i-a-1}}$$
and its image is determined by the following equations
$\{F_p^{\beta_0}(X_i,X_{i+1})=0\}_{0 \leq i \leq n}.$
\end{corollary}
\begin{proof}
Using Corollaries \ref{emb3} and \ref{key3}, we obtain the required assertions.
\end{proof}

\section{Review of the irreducible components in the reduction of the supersingular locus 
$W_A(p^3)$
 in \cite{CM}}

\noindent
In this section, we review the main part of the work of \cite{CM}, namely
the calculations of the ``bridging component''
$\overline{\mathbf{Z}}^A_{1,1}$ and the ``new components''
defined by $a^p-a=s^2.$
See \cite[Sections 6,7 and 8]{CM}.
The new components attach to the curve $\overline{\mathbf{Z}}^A_{1,1}$
at distinct singular points in the stable reduction of the modular curve $X_0(p^3).$
Coleman-McMurdy calculated the new components using the Gross-Hopkins theory.
In this section, we recalculate the bridging component and the new components
in an elementary manner with using the identification in Corollary \ref{prod}.
However, our calculation of the bridging component
is essentially the same as Coleman-McMurdy's one.

\subsection{The bridging component in the stable reduction of $X_0(p^3)$
in \cite[Section 8]{CM}}

By using Corollary \ref{key} and Corollary \ref{prod},
we recalculate ``the bridging component'' in the stable reduction
of $X_0(p^3),$ i.\ e.\ the reduction of $\mathbf{Z}^A_{1,1},$
which is calculated by Coleman-McMurdy in \cite[Section 7]{CM}.
All results in this subsection are already proved in loc.\ cit.

Let $\alpha, \gamma$ be elements
such that $\alpha^{2p}=p$
and
$\gamma^p=p/\alpha^{p+1}.$
We have $v(\alpha)=1/2p,v(\gamma)=(p-1)/2p^2.$
If we put $\alpha_1:=\alpha/\gamma,$
we have $\alpha=\alpha_1^p,$
$\gamma=\alpha_1^{p-1}$ and
$v(\alpha_1)=1/2p^2.$

Fix a supersingular elliptic curve $A/\mathbb{F}_p$
with $j(A) \neq 0,1728.$
We recall the identification of $\mathbf{Z}^A_{1,1}$
in Corollary \ref{prod}
$$\mathbf{Z}^A_{1,1}
\simeq \{(\{X_i\}_{0 \leq i \leq 3}) \in \mathbf{C}^{A,0}_{1/2p} \times 
\mathbf{C}^{A,0}_{1/2} \times \mathbf{C}^{A,0}_{1/2} \times \mathbf{C}^{A,0}_{1/2p}
|\ F_p^{\beta_0}(X_i,X_{i+1})=0\ (0 \leq i \leq 2)\}$$
and rewrite parameters as follows $X,U,V$ and $Y.$
We change variables as follows
$X=\alpha x, U=\alpha^p u, V=\alpha^{p} v$ and $Y=\alpha y,$ 
with $x,u,v$ and $y$ invertible functions.
Now, we simply write
$F_p(X,Y)$ for $F_p^{\beta_0}.$

\begin{proposition}(\cite[Proposition 8.2]{CM})\label{p1}
Over $R=\mathbb{Z}_p[\alpha_1] \otimes \mathbb{Z}_{p^2},$
the reduction
of $\mathbf{Z}^A_{1,1}$
is defined by the following equation
$$Z^p+\bar{c}_0\biggl(\frac{\bar{c}_0}{x^{p+1}}+\frac{x^{p+1}}{\bar{c}_0}\biggr)=0.$$
Hence, over $R,$
its reduction
is a reduced, connected, affine curve
of genus $0$ with only one branch through each sungular point $x=\zeta$
with $\zeta^{2(p+1)}=\bar{c}^2_0.$
\end{proposition}
\begin{proof}
The equation $F_p(\alpha x,\alpha^p u)=0$ induces the following
congruences by Corollary \ref{key}.1
\begin{equation}\label{w1}
u=x^p+\frac{pc_0}{\alpha^{p+1}}\frac{1}{x}+\alpha^pc_1,
v=y^p+\frac{pc_0}{\alpha^{p+1}}\frac{1}{y}+\alpha^pc_1\ ({\rm mod}\ \alpha^{p+1}).
\end{equation}
The equation $F_p(\alpha^p u,\alpha^p v)=0$
induces the following congruences
by Corollary \ref{key}.2 and (\ref{w1})
\begin{equation}\label{w2}
uv=c_0+c_1\alpha^p(u+v)=c_0+c_1 \alpha^{p}(x^p+y^p)\ ({\rm mod}\ \alpha^{p+1}).
\end{equation}
By (\ref{w1}), we obtain the following
$$uv=x^py^p+\frac{pc_0}{\alpha^{p+1}}\biggl(\frac{x^p}{y}+\frac{y^p}{x}\biggr)
+c_1\alpha^p(x^p+y^p)\ ({\rm mod}\ \alpha^{p+1}).$$
By this congruence and (\ref{w2}), the following congruence holds
\begin{equation}\label{w3}
x^py^p+\frac{pc_0}{\alpha^{p+1}}\biggl(\frac{x^p}{y}+\frac{y^p}{x}\biggr)
=c_0\ ({\rm mod}\ \alpha^{p+1}).
\end{equation}
We introduce a new parameter $Z$ as follows,
 as in \cite[the proof of Proposition 8.2]{CM}
$$
xy=c_0+\gamma Z.
$$
Substituting this to 
the term $x^py^p$ in the left hand side of the congruence 
(\ref{w3}) and dividing it by $(p/\alpha^{p+1})=\alpha^{p-1}$,
we acquire the following
\begin{equation}\label{w5}
Z^p+c_0\biggl(\frac{x^p}{y}+\frac{y^p}{x}\biggr) \equiv
0\ ({\rm mod}\ \alpha^2).
\end{equation}
The required assertion follows from (\ref{w5})
and $xy=\bar{c}_0.$
\end{proof}

\begin{remark}
To calculate the reduction of $\mathbf{Z}^A_{1,1},$
it suffices to consider (\ref{w1}) modulo $\alpha^p.$
However, we need the congruence
(\ref{w1}) to compute the ``new components''
in the stable reduction of $X_0(p^3),$
in the next subsection.
\end{remark}
\begin{remark}
We assume that $j(A) \neq 0, 1728$
in Proposition \ref{p1}.
By \cite[Proposition 4.2]{CM},
as remarked in \cite[Remark 8.8]{CM},
similar results now follow for {\it any} other supersingular
elliptic curve $A'.$
The ``bridging component''
in the supersingular locus $W_{A'}(p^3) \subset X_0(p^3)$
has the following equation
$$Z^p+\bar{c}_0\biggl(\frac{X^{(p+1)/i(A')}}{\bar{c}_0}+\frac{\bar{c}_0}{X^{(p+1)/i(A')}}\biggr)=0.$$
See loc.\ cit.\ for more detail.
\end{remark}

\subsection{The new components in the stable reduction of $X_0(p^3)$ in \cite[subsection 8.2]{CM}}
``The new components'' in the stable reduction of $X_0(p^3),$
defined by the Artin-Schreier equation $a^p-a=s^2,$
were found by using a moduli-theoretic interpretation of $X_0(p^3)$
in \cite[subsection 8.2]{CM}.
The components attach to the reduction of $\mathbf{Z}^A_{1,1}$
at several distinct points in the stable reduction of $X_0(p^3).$
To deduce the Artin-Schreier equation above, Coleman-McMurdy 
construct several involutions on the space $\mathbf{Z}^A_{1,1},$
using the Woods Hole Theory and the Gross-Hopkins theory.
See
\cite[Section 8]{CM}
for more detail.
In this subsection, we recalculate the defining equations
of the new components in a more elementary and explicit manner
without using the above arithmetic theories.

We keep the same notations as in the previous subsection.
Dividing the congruence (\ref{w5})
by $(c_0+\gamma Z)^{p},$
we obtain the following congruence
\begin{equation}\label{w6}
\biggl(\frac{Z}{c_0+\gamma Z}\biggr)^p+
c_0\biggl(\frac{1}{x^{p+1}}+\bigl(\frac{x}{c_0+\gamma Z}\bigr)^{p+1}\biggr)
\equiv 0\ ({\rm mod}\ \alpha^2).
\end{equation}
We set $F(Z,x):=c_0\bigl(\frac{1}{x^{p+1}}+\bigl(\frac{x^{p+1}}{c_0+\gamma Z}\bigr)^{p+1}\bigr).$ 

We choose a root $\gamma_{0,+}$
(resp.\ $\gamma_{0,-}$)
of the following equation
$z^p+2c_0(c_0+\gamma z)^{(p-1)/2}=0.$
(resp.\ $z^p-2c_0(c_0+\gamma z)^{(p-1)/2}=0.$) 
We fix roots $(c_0+\gamma_{0,\pm} \gamma)^{1/2}.$
For each $(p+1)$-th root of unity $\zeta,$
(resp.\ each $(p+1)$-th root of $-1$ $\zeta,$)
we set $x_{0,+,\zeta}:=\zeta(c_0+\gamma_{0,+} \gamma)^{1/2}.$
(resp.\ $x_{0,-,\zeta}:=\zeta(c_0+\gamma_{0,-} \gamma)^{1/2}.$)
For simplicity,
we write $\gamma_0$
for $\gamma_{0,\pm}$
and $x_0$ for $x_{0,\pm, \zeta}$
respectively.
Then, we have an equality
$\bigl(\frac{\gamma_0}{c_0+\gamma_0 \gamma}\bigr)^p+F(\gamma_0,x_0)=0.$

\begin{lemma}\label{l1}
Let the notations be as above.
Then, the followings hold
\begin{itemize}
\item $\partial_xF(\gamma_0,x_0)=0.$\\
\item $v(\partial_ZF(\gamma_0,x_0))=v(\gamma)=(p-1)/2p^2.$\\
\item $\partial^2_xF(\gamma_0,x_0)$ is a unit.\\
\end{itemize}
\end{lemma}
\begin{proof}
The assertions follow from direct computations. We omit the detail.
\end{proof}

We choose elements $\gamma_1, \alpha_1$
such that
$$\gamma_1^{p-1}=-c_0^p \partial_ZF(\gamma_0,x_0), 
\alpha_1^2=-\biggl(\frac{\gamma_1}{c_0}\biggr)^p\biggl(\frac{1}{2}\partial^2_xF(\gamma_0,x_0)\biggr)^{-1}.$$
By Lemma \ref{l1},
we have $v(\gamma_1)=1/2p^2,v(\alpha_1)=1/4p.$

\begin{lemma}\label{l2}
Let the notation be as above.
We change variables as follows
$Z=\gamma_0+\gamma_1 a$ and $x=x_0+\alpha_1 s.$
Then, the following congruence holds
\begin{equation}\label{eq1}
F(\gamma_0+\gamma_1 a,x_0+\alpha_1 s)
=F(\gamma_0,x_0)+\frac{1}{2}\partial^2_xF(\gamma_0,x_0)(\alpha_1 s)^2
+\partial_ZF(\gamma_0,x_0)\gamma_1 a\ ({\rm mod}\ \gamma \gamma_1^2)
\end{equation}
\end{lemma}
\begin{proof}
By Lemma \ref{l1}.1, it suffices to show that $\partial_x\partial_ZF(\gamma_0,x_0)\gamma_1\alpha_1
\equiv 0\ ({\rm mod}\ \gamma \gamma_1^2).$
Since we have $v(\partial_x\partial_ZF(\gamma_0,x_0))
=v(\gamma)$ and $v(\alpha_1)>v(\gamma_1),$
the required assertion follows.
\end{proof}

\begin{corollary}(\cite[Proposition 8.7]{CM})\label{go1}
For each supersingular elliptic curve $A/\mathbb{F}_p$
with $j(A) \neq 0,1728,$
there exist
$2(p+1)$ components in the stable reduction of $X_0(p^3),$
defined by the following equation
$$a^p-a=s^2.$$
These components attach to the reduction of $\mathbf{Z}^A_{1,1}$
at its singular points.
\end{corollary}
\begin{proof}
We have the equality
$\bigl(\frac{\gamma_0}{c_0+\gamma_0 \gamma}\bigr)^p+F(\gamma_0,x_0)=0$
as mentioned
above.
Note that $v(\gamma \gamma_1^2)=(p+1)/2p^2<v(\alpha^2)=1/p.$
Therefore, by Lemma \ref{l2}, the congruence
(\ref{w6}) induces the following
congruence under the variables $(a,s)$
$$\biggl(\frac{\gamma_1}{c_0}\biggr)^pa^p+\frac{1}{2}\partial^2_x
F(\gamma_0,x_0)\alpha_1^2s^2+\partial_ZF(\gamma_0,x_0)\gamma_1a \equiv 0\ ({\rm mod}\ \gamma \gamma_1^2).$$
Hence, by the definitions of $\gamma_1,\alpha_1,$
we acquire
$$\biggl(\frac{\gamma_1}{c_0}\biggr)^p \times (a^p-a-s^2) \equiv 0\ ({\rm mod}\ \gamma \gamma_1^2).$$
Note that $v(\gamma_1^p)=1/2p<v(\gamma \gamma_1^2)=(p+1)/2p^2.$
Dividing the above congruence by $\bigl(\frac{\gamma_1}{c_0}\bigr)^p,$
we obtain the Artin-Schreier equation $a^p-a=s^2.$
\end{proof}

\begin{remark}\label{q1}
The intersection multiplicities for the stable reduction of $X_0(p^3)$
were almost calculated in \cite[Tha table $1$ in subsection 9.1]{CM}.
However, the intersection multiplicities of the bridging component with
the new components were {\it not}
calculated in loc.\ cit.
In this remark, we give the width of each singular residue class
of $\mathbf{Z}^A_{1,1}.$
See \cite[Corollary 7.3]{CM} for the singular residue classes of $\mathbf{Z}^A_{1,1}$.
See loc.\ cit.\ or subsection 5.1
for a relationship between a width and an intersection multiplicity.
We give a complete list of intersection multiplicity data for $X_0(p^3)$
in Table 4 in subsection 5.4.

By putting $Z=\gamma_0+z,x=x_0+t$ in (\ref{w6}), we acquire
the following congruence
\begin{equation}\label{remm}
\biggl(\frac{z}{c_0}\biggr)^p
+\partial_ZF(\gamma_0,x_0)z+\frac{1}{2}\partial_x^2F(\gamma_0,x_0)t^2+f(z,t)=0\ ({\rm mod}\ \alpha^2)
\end{equation}
where $f(z,t)=\sum_{(i,j) \neq (0,0),(1,0),(0,1),(0,2)}\frac{1}{i!j!}\partial_Z^i\partial_x^j
F(\gamma_0,x_0)z^it^j.$
By (\ref{remm}), the defining equation of $\mathbf{Z}^A_{1,1}$
is written as follows
$$\biggl(\frac{z}{c_0}\biggr)^p
+\partial_ZF(\gamma_0,x_0)z+\frac{1}{2}\partial_x^2F(\gamma_0,x_0)t^2+f(z,t)=\alpha^2 G(z,t)$$
with some function $G(z,t).$
Let $S \subset \mathbf{Z}^A_{1,1}$
be a singular residue class. See \cite[Corollary 7.3]{CM} for more detail.
By the computation above, we obtain an explicit description
of $S,$ and its underlying affinoid $\mathbf{X}_S \subset S$
as follows
$$S(\mathbb{C}_p)
=
\{(z,t) \in \mathfrak{m}_{\mathbb{R}_p} \times \mathfrak{m}_{\mathbb{R}_p}
|\ \biggl(\frac{z}{c_0}\biggr)^p
+\partial_ZF(\gamma_0,x_0)z+\frac{1}{2}\partial_x^2F(\gamma_0,x_0)t^2+f(z,t)=\alpha^2 G(z,t)\},$$
$$\mathbf{X}_S(\mathbb{C}_p)=\{(z,t) \in S(\mathbb{C}_p)|v(z) \geq v(\gamma_1),v(t) \geq v(\alpha_1)\}.$$
Hence, the complement $S \backslash \mathbf{X}_S$
is isomorphic to an annulus
$x \in A(p^{-1/4p^2},1)$ by a map
$z=x^2+\cdots,t=x^p+\cdots.$
Therefore, we obtain the width
$1/4p^2$ of the annulus $S \backslash \mathbf{X}_S.$

\end{remark}

\section{Irreducible components appearing in the reduction of the supersingular
locus $W_A(p^4)$}
{\rm
 In this section, fixing a supersingualr elliptic curve
$A/\mathbb{F}_p$ with $j(A) \neq 0, 1728,$
we will calculate the defining equations of all irreducible components
in the reduction of the supersingular locus $W_A(p^4),$
namely the reductions of $\mathbf{Z}^A_{2,1},\mathbf{Z}^A_{1,2},\mathbf{Y}^A_{3,1},\mathbf{Y}^A_{1,3}$
and $\mathbf{Y}^A_{2,2}.$
The main part in this section is in calculating the reduction of $\mathbf{Y}^A_{2,2}$
in Corollary \ref{good}.
Further, we analyze the singular residue classes in $\mathbf{Y}^A_{2,2}.$
In Corollary \ref{op}, we prove 
that there exist $(p+1)$ components, defined by $a^p-a=t^{p+1},$
 in the stable reduction
 of $W_A(p^4).$
These components attach to the reduction of $\mathbf{Y}^A_{2,2}$
at $(p+1)$ singular points.
This is a new phenomenon which appears in the stable reduction of $X_0(p^4).$ 
The curve defined by $a^p-a=t^{p+1}$ is called the Deligne-Lusztig curve
for ${\rm SL}_2(\mathbb{F}_p).$

}

\subsection{The reductions of the subspaces $\mathbf{Z}^A_{2,1}$ and $\mathbf{Z}^A_{1,2}$
in $W_A(p^4) \subset X_0(p^4)$}

{\rm In this subsection, we calculate the reductions of $\mathbf{Z}^A_{2,1}, \mathbf{Z}^A_{1,2}
\subset W_A(p^4).$
These calculations are very similar to 
the ones in Section 3.  

Let $\alpha_2$ be an element such that $\alpha_2^{2p^3}=p.$
We set $\alpha_1:=\alpha_2^p$ and $\alpha:=\alpha_1^p.$
Further, we put $\gamma_1:=\alpha_2^{p-1}$
 and $\gamma:=\gamma_1^p.$
 Then, we have $\gamma^p=p/\alpha^{p+1},$
 $v(\alpha)=1/2p,v(\alpha_1)=1/2p^2$ and $v(\alpha_2)=1/2p^3.$
 Clearly, we have $v(\gamma_1)=(p-1)/2p^3$
 and $v(\gamma)=(p-1)/2p^2.$

As in the previous section, we fix a supersingular elliptic curve
$A/\mathbb{F}_p$ with $j(A) \neq 0,1728.$
We recall the identification of $\mathbf{Z}^A_{2,1}$ in Corollary \ref{prod}
$$\mathbf{Z}^A_{2,1} \simeq \{(\{X_i\}_{0 \leq i \leq 4}) \in 
\mathbf{C}^{A,0}_{1/2p^2} \times \mathbf{C}^{A,0}_{1/2p}
\times \mathbf{C}^{A,0}_{1/2} \times \mathbf{C}^{A,0}_{1/2} 
\times \mathbf{C}^{A,0}_{1/2p} |\ F_p^{\beta_0}(X_i,X_{i+1})=0\ (0 \leq i \leq 3)\}$$
 and rewrite the variables as follows $X_1,X,U,V,Y.$
 We change variables as follows
 $X_1=\alpha_1 x_1, X=\alpha x,U=\alpha^p u, V=\alpha^p v$
 and $Y=\alpha y$
with $x_1,x,u,v,y$ invertible functions.
We simply write $F_p$ for $F_p^{\beta_0}.$

Similarly as above, we consider the identification of $\mathbf{Z}^A_{1,2}$
in Corollary \ref{prod}
$$\mathbf{Z}^A_{2,1} \simeq \{(\{X_i\}_{0 \leq i \leq 4}) \in 
\mathbf{C}^{A,0}_{1/2p} \times \mathbf{C}^{A,0}_{1/2}
\times \mathbf{C}^{A,0}_{1/2} \times \mathbf{C}^{A,0}_{1/2p} 
\times \mathbf{C}^{A,0}_{1/2p^2} |\ F_p(X_i,X_{i+1})=0\ (0 \leq i \leq 3)\}$$
By the symmetry of $F_p$ in Lemma \ref{0},
 the reduction of $\mathbf{Z}^A_{1,2}$
is equal to the reduction of $\mathbf{Z}^A_{2,1}.$
Hence, we only calculate the reduction of $\mathbf{Z}^A_{2,1}.$ 
}{\rm
\begin{lemma}
Over $R:=\mathbb{Z}_p[\alpha_2] \otimes \mathbb{Z}_{p^2},$
the reduction of $\mathbf{Z}^A_{2,1}$
is defined by the following equation
$$Z_1^p+\bar{c}_0\biggl(\frac{\bar{c}_0}{x_1^{p+1}}+\frac{x_1^{p+1}}{\bar{c}_0}\biggr)=0.$$
\end{lemma}
\begin{proof}
As in the proof of Proposition \ref{p1},
we acquire the followings by the equations  $F_p(\alpha x,\alpha^p u)=0$
and $F_p(\alpha^p v,\alpha y)=0$
\begin{itemize}
\item $xy=c_0+\gamma Z$ \\
\item $Z^p+c_0(\frac{x^{p+1}}{c_0+\gamma Z}+\frac{c_0}{x^{p+1}})=0\ ({\rm mod}\ \alpha^2). $\\
\end{itemize}
By Corollary \ref{key}.1,
the equation $F_p(\alpha_1 x_1,\alpha x)=0$
induces the following $x=x_1^p\ ({\rm mod}\ \alpha^2).$
Therefore, we obtain the following congruence by 2
\begin{equation}\label{ioh}
Z^p+c_0\biggl(\frac{x_1^{p(p+1)}}{c_0+\gamma Z}+\frac{c_0}{x_1^{p(p+1)}}\biggr)=0\ ({\rm mod}\ \alpha^2).
\end{equation}
Since we have $(c_0+\gamma Z)^{-1} \equiv 1/c_0-\gamma Z/c_0^2$ modulo
$\gamma^2$ and $v(\gamma)<v(\alpha),$
the above congruence (\ref{ioh}) is rewritten as follows
\begin{equation}\label{41}
Z^p+c_0\biggl(\frac{x_1^{p(p+1)}}{c_0}+\frac{c_0}{x_1^{p(p+1)}}\biggr)
=\gamma\frac{x_1^{p(p+1)}}{c_0}Z\ ({\rm mod}\ \gamma^2).
\end{equation}
We introduce a new parameter $Z_1$
as follows
\begin{equation}\label{42}
Z+c_0\biggl(\frac{x_1^{p+1}}{c_0}+\frac{c_0}{x_1^{p+1}}\biggr)
=\gamma_1\frac{x_1^{p+1}}{c_0}Z_1.
\end{equation}
Substituting this to (\ref{41}) and dividing it by $\gamma,$
we obtain $Z_1^p=Z\ ({\rm mod}\ \gamma).$
Again substituting this to (\ref{42}), the following congruence holds
$$Z_1^p+c_0\biggl(\frac{x_1^{p+1}}{c_0}(1-\gamma_1Z_1/c_0)+\frac{c_0}{x_1^{p+1}}\biggr)=0\ ({\rm mod}\ \gamma).$$
Since we have $1-\gamma_1 Z_1/c_0=(1+\gamma_1Z_1/c_0)^{-1}$ modulo
$\gamma_1^2,$ we acquire the following congruence,
which is similar to (\ref{w6}),
\begin{equation}\label{ww1}
\biggl(\frac{Z_1}{c_0+\gamma_1Z_1}\biggr)^p
+c_0\biggl(\frac{1}{x_1^{p+1}}+\bigl(\frac{x_1}{c_0+\gamma_1Z_1}\bigr)^{p+1}\biggr)=0\
({\rm mod}\ \gamma_1^2)
\end{equation}
Hence, the required assertion follows.
\end{proof}

\begin{proposition}
For each supersingular elliptic curve $A/\mathbb{F}_p$ with $j(A) \neq 0,1728,$
there exist $2(p+1)$ irreducible components, which attach to the component 
$\overline{\mathbf{Z}}^A_{2,1}$ in the stable reduction of $X_0(p^4),$
defined by the following equation
$$a^a-a=s^2.$$
\end{proposition}
\begin{proof}
We deduce the assertion from the congruence (\ref{ww1}) in the same way as
Corollary \ref{go1} is deduced from the congruence (\ref{w6}) in the previous section.
\end{proof}
\begin{remark}
Let $S \subset \mathbf{Z}^A_{2,1}$ denote the singular residue class.
This space and its underlying affinoid $\mathbf{X}_S$
have similar descriptions to the ones in Remark \ref{q1}.
The complement $\mathbf{X}_S \backslash S$
is an annulus with width $1/4p^3.$
\end{remark}

\subsection
{Edixhoven's horizontal component $\overline{\mathbf{Y}}^A_{1,1}$ 
in the stable reduction of $X_0(p^2)$ and the reductions of $\mathbf{Y}^A_{3,1}$
and $\mathbf{Y}^A_{1,3}$ in $W_A(p^4)$}

In this subsection, we explicitly calculate the components $\overline{\mathbf{Y}}^A_{3,1}$
and $\overline{\mathbf{Y}}^A_{1,3},$
appearing in the stable reduction of $X_0(p^4).$
In this process, we recalculate Edixhoven's
horizontal component, i.\ e.\ the reduction 
of $\mathbf{Y}^A_{1,1}$
in the stable reduction of $X_0(p^2)$
 which was originally found by Edixhoven in \cite[Section 2]{E}
 and refound by Coleman-McMurdy on the basis of the rigid geometry in \cite[Section 5]{CM}.
 
Let $\alpha$ be an element such that $\alpha^{p(p+1)}=p$
and $\alpha_1$ an element such that $\alpha_1^p=\alpha.$
We have $v(\alpha_1)=1/p^2(p+1),v(\alpha)=1/p(p+1).$
As in the previous subsection,
we fix a supersingular elliptic curve $A/\mathbb{F}_p$
with $j(A) \neq 0, 1728.$
We recall the identification 
of $\mathbf{Y}^A_{3,1}$ in Corollary \ref{prod} 
$$\mathbf{Y}^A_{3,1} \simeq \{(\{X_i\}_{0 \leq i \leq 3}) \in 
\mathbf{C}^{A,0}_{1/p^2(p+1)} \times \mathbf{C}^{A,0}_{1/p(p+1)} \times 
\mathbf{C}^{A,0}_{1/(p+1)}
\times
\mathbf{C}^{A,0}_{\geq p/(p+1)}
\times
\mathbf{C}^{A,0}_{1/(p+1)}
$$
$$|\ F_p^{\beta_0}(X_i,X_{i+1})=0\ (0 \leq i \leq 3), X_2 \neq X_4\}$$
and rewrite parameters as follows $X_1,X,U,W,V.$
We change variables as follows
$X_1=\alpha_1x_1, X=\alpha x, U=\alpha^p u, W=\alpha^{p^2} w, V=\alpha^p v$
with $x_1,x,u$ and $v$
invertible functions. 

We simply write $F_p$ for $F_p^{\beta_0}.$
By the symmetry of $F_p$
in Lemma \ref{0}, the reduction of $\mathbf{Y}^A_{3,1}$
is equal to the one of $\mathbf{Y}^A_{1,3}.$ 
Therefore, we only consider the reduction of $\mathbf{Y}^A_{3,1}.$
First, we consider the equations
 $F_p(\alpha^p u,\alpha^{p^2}w)=0$ and
$F_p(\alpha^{p^2}w, \alpha^p v)=0.$
Let the notation be as in subsection 2.1.
Further, we write 
$h(X)=c_1+Xh_1(X)$
and $X^p-Y^p=(X-Y)^p+pf(X,Y).$
The equation $F_p(\alpha^p u,\alpha^{p^2}w)=0$ induces the following congruence
by Corollary \ref{prod}
\begin{equation}\label{y1}
w=u^p+\frac{c_0}{u}+\alpha^p(g(\alpha^pu)+(\alpha^pu)^{p-1}h(\alpha^p u))
+\frac{c_0h(\alpha^p u)}{u^2}\alpha^{p^2}+pH(u)\ ({\rm mod}\ p\alpha)
\end{equation}
where $H(u)=h(\alpha^pu)g(\alpha^pu)/u.$
Rewriting $g(\alpha^pu)$ for $g(\alpha^pu)+(\alpha^pu)^{p-1}h(\alpha^p u),$
and $H(u)$
for $H(u)+\frac{c_0h_1(\alpha^pu)}{u},$
the congruence (\ref{y1})
has the following form
\begin{equation}\label{y2}
w=u^p+\frac{c_0}{u}
+\alpha^pg(\alpha^pu)+\frac{c_0c_1}{u^2}\alpha^{p^2}
+pH(u)\ ({\rm mod}\ p\alpha).
\end{equation}
\begin{lemma}\label{y4}
Let the notation be as above.
Then, the following relationship between $u$ and $v$ holds
\begin{equation}\label{y3}
uv(u-v)^{p-1}-c_0+\alpha^p\frac{g(\alpha^pu)-g(\alpha^pv)}{u-v}uv
-c_0c_1\alpha^{p^2}\biggl(\frac{u+v}{uv}\biggr)+pH(u,v)=\ ({\rm mod}\ p\alpha)
\end{equation}
where $H(u,v)=
\frac{H(u)-H(v)+f(u,v)}{u-v}uv.$
\end{lemma}

\begin{proof}
By $U \neq V,$ we acquire $u \neq v.$
Hence, the assertion follows from (\ref{y2})
and the same relationship between $w$ and $v.$
\end{proof}
\begin{corollary}(\cite[Theorem 2.1.1]{E} or \cite[Proposition 5.2]{CM})\label{y5}
Let $K$ be an extension of $\mathbb{Q}_{p^2}$
such that $(p+1)|e_K.$
Over $R_K,$ the reduction of $\mathbf{Y}^A_{1,1}$
is a smooth affine curve of genus $(p-1)/2$
with $4$ points at infinity with the following equation
$$uv(u-v)^{p-1}=\bar{c}_0.$$
\end{corollary}
\begin{proof}
We consider the embedding
of $\mathbf{Y}^A_{1,1}$ in Corollary \ref{prod}
$$\mathbf{Y}^A_{1,1} \simeq \{(\{X_i\}_{0 \leq i \leq 2}) \in \mathbf{C}^{A,0}_{1/(p+1)} \times 
\mathbf{C}^{A,0}_{\geq p/(p+1)} 
\times 
\mathbf{C}^{A,0}_{1/(p+1)}|\ 
F_p(X_i,X_{i+1})=0\ (0 \leq i \leq 1),X_0 \neq X_2\}$$
and write parameters as follows $U,W,V.$
We change variables $U=\alpha^p u,$$W=\alpha^{p^2} w$
and $V=\alpha^p v$ with $\alpha$ as above.
The assertion follows from
Lemma \ref{y4}.
\end{proof}
\begin{remark}
To deduce Corollary \ref{y5}, it is sufficient to consider (\ref{y2})
modulo $\alpha^{p+1}.$
However, we need to consider (\ref{y2}) to compute the reduction of $\mathbf{Y}^A_{2,2}$
later.
\end{remark}
We compute the reduction of $\mathbf{Y}^A_{3,1}.$
\begin{corollary}\label{y6}
Let $K$ be an extension of $\mathbb{Q}_{p^2}$ such that $p^2(p+1)|e_K.$
Over $R_K,$ 
the reduction of $\mathbf{Y}^A_{3,1}$ is a smooth affine curve of genus $(p-1)/2$
with $4$ points at infinity with the following equation
$$x_1^{p^2}y(x_1^{p^2}-y)^{p-1}=\bar{c}_0.$$ 
\end{corollary}

\begin{proof}
The equations $F_p(\alpha_1x_1,\alpha x)=0$ and $F_p(\alpha x,\alpha^p u)=0$
induce the following congruences by Corollary \ref{prod}.1,
$x=x_1^p\ ({\rm mod}\ \alpha_1)$ and $u=x^p\ ({\rm mod}\ \alpha_1).$
Thereby, the assertion follows from (\ref{y3}) immediately.
\end{proof}

\begin{remark}
We consider the spaces 
$\mathbf{Y}^A_{2,1}, \mathbf{Y}^A_{1,2} \subset X_0(p^3)$
which corresponds to the spaces $\mathbf{E}^A_1,\mathbf{E}^A_{2}$
in \cite[Section 6]{CM}
respectively.
We can compute the reductions $\overline{\mathbf{Y}}^A_{2,1}$
and $\overline{\mathbf{Y}}^A_{1,2}$
in the stable reduction of $X_0(p^3)$ in the same way as in Corollary \ref{y6}.
We write down the equations of them.
The reductions $\overline{\mathbf{Y}}^A_{2,1}$
and $\overline{\mathbf{Y}}^A_{1,2}$ are defined by the following equation
$$x^py(x^p-y)^{p-1}=\bar{c}_0.$$
See also \cite[Remark 9.3]{CM}.
\end{remark}
\begin{remark}
In the same way as in \cite[Corollary 5.4]{CM},
for {\it any} supersingular elliptic curve $A$,
the reductions of $\mathbf{Y}^A_{2,1},
 \mathbf{Y}^A_{1,2} \subset W_A(p^3)$ must have the following equations
 $$y^{(p+1)/i(A)}+1=x^2, z^p=(x+1)^{i(A)}/y$$
 and genus $(p+1)/2i(A)-1.$
 In the same way as above, for any supersingular elliptic curve $A,$
 the reductions of $\mathbf{Y}^A_{3,1},\mathbf{Y}^A_{1,3} \subset W_A(p^4)$
 must have the following equations
 $$y^{(p+1)/i(A)}+1=x^{2}, z^{p^2}=(x+1)^{i(A)}/y$$
and genus $(p+1)/i(A)-1$.
See loc.\ cit.\ for more detail.
\end{remark}
\subsection{The new bridging component, i.\ e.\ the reduction of $\mathbf{Y}^A_{2,2}$}

In this subsection, we explicitly calculate the component $\overline{\mathbf{Y}}^A_{2,2},$
which we call the ``new bridging component''
in the stable reduction of $X_0(p^4).$
The projective completion of the component $\overline{\mathbf{Y}}^A_{2,2}$
intersects the projective completions of the curves 
$\overline{\mathbf{Z}}^A_{3,1}, \overline{\mathbf{Z}}^A_{1,3}$
and the two ordinary components $\overline{W}^{\pm}_{2,2}$
at ordinary double points.
This fact is proved later in subsection 5.2.
The component $\overline{\mathbf{Y}}^A_{2,2}$
is defined by the following equations
$$xy(x-y)^{p-1}=\bar{c}_0,
Z^p+\bar{c}_0+\bar{c}^2_0\biggl(\frac{1}{x^{p+1}}+\frac{1}{y^{p+1}}\biggr)=0.$$
The curve  $\overline{\mathbf{Y}}^A_{2,2},$
has $(p+1)$ singular points at $(x,y)=(-\zeta,\zeta)$
with $\zeta^{p+1}=-\bar{c}_0.$ 
See Corollary \ref{good} for more detail.

Let $\alpha$ be an element such that $\alpha^{p(p+1)}=p.$
Let $\gamma$ be an element satisfying $\gamma^p=p/\alpha^{p+1}.$
We have $v(\alpha)=1/p(p+1),v(\gamma)=(p-1)/p^2.$
If we put $\alpha_1:=\alpha^p/\gamma$, we have $\alpha_1^p=\alpha,$
 $\alpha_1^{p^2-1}=\gamma$ and $v(\alpha_1)=1/p^2(p+1).$
As in the previous subsection, we fix a supersingular elliptic curve
$A/\mathbb{F}_p$
with $j(A) \neq 0,1728.$
We recall the identification of $\mathbf{Y}^A_{2,2}$
in Corollary \ref{prod}.1
$$\mathbf{Y}^A_{2,2} \simeq \{(\{X_i\}_{0 \leq i \leq 4}) \in \mathbf{C}^{A,0}_{1/p(p+1)} \times 
\mathbf{C}^{A,0}_{1/(p+1)} \times 
\mathbf{C}^{A,0}_{\geq p/(p+1)}  \times
\mathbf{C}^{A,0}_{1/(p+1)}  \times \mathbf{C}^{A,0}_{1/p(p+1)}|$$
$$F_p(X_i,X_{i+1})=0\ (0 \leq i \leq 3), X_1 \neq X_3\}$$
and rewrite parameters as follws $X,U,W,V,Y.$
We change variables
as follows
$X=\alpha x,U=\alpha^p u,W=\alpha^{p^2} w, V=\alpha^p v$ and $Y=\alpha y$
with $x,u,v,w$ and $y$ invertible functions.
We simply write $F_p$ for $F_p^{\beta_0}.$
We have considered the equations 
$F_p(\alpha^p u,\alpha^{p^2} w)=0$
and $F_p(\alpha^{p^2} w,\alpha^p v)=0,$
and deduced the relationship between $u$ and $v$
in Lemma \ref{y4}.

In the following, we consider the equation
$F_p(\alpha x,\alpha^p u)=0.$
It gives the following congruence by Corollary \ref{key}.1
$$u=x^p+\frac{pc_0}{\alpha^{p+1}}\frac{1}{x}+\frac{p}{\alpha^p}(g(\alpha x)+
(\alpha x)^{p-1}h(\alpha x))\ 
({\rm mod}\  p\alpha).$$
Rewriting $g(\alpha x)$ for $g(\alpha x)+(\alpha x)^{p-1}h(\alpha x)$
as in (\ref{y2}),
the above congruence has the following
form
\begin{equation}\label{yy1}
u=x^p+\frac{pc_0}{\alpha^{p+1}}\frac{1}{x}+\frac{p}{\alpha^p}g(\alpha x)\ ({\rm mod}\ p\alpha).
\end{equation}
By the same reason, 
the following relationship between $y$ and $v$ holds
\begin{equation}\label{yy2}
v=y^p+\frac{pc_0}{\alpha^{p+1}}\frac{1}{y}+\frac{p}{\alpha^p}g(\alpha y)\ ({\rm mod}\ p\alpha).
\end{equation}
\begin{lemma}\label{yyq1}
We put $h(x,y):=f(x,y)(x-y)^{p(p-2)}x^py^p$
and
$H(x,y,u,v):=H(u,v)-h(x,y)$
for which we simply write
$H.$ We set $\mathcal{Z}:=xy(x-y)^{p-1}.$
Then, there exists the following relationship between $x,y,u$ and $v$
$$\mathcal{Z}^p+
\frac{pc_0}{\alpha^{p+1}}\mathcal{Z}^{p-1}+
\frac{p}{\alpha^p}x^py^p(x-y)^{p(p-2)}(g(\alpha y)-g(\alpha x))$$
$$+\frac{pc_0}{\alpha^{p+1}}\bigg(\frac{1}{x^{p+1}}+\frac{1}{y^{p+1}}\biggr)
\mathcal{Z}^{p}+\frac{p}{\alpha^p}\{y^pg(\alpha x)+x^pg(\alpha y)\}(x-y)^{p(p-1)}$$
$$=c_0-\alpha^p\frac{g(\alpha^p u)-g(\alpha^p v)}{u-v}uv+
c_0c_1\alpha^{p^2}\biggl(\frac{x+y}{xy}\biggr)^p-pH\ ({\rm mod}\ p\alpha).$$
\end{lemma}
\begin{proof}
In this proof, we consider every congruence modulo $p\alpha.$
By (\ref{yy1}) and (\ref{yy2}), we obtain the following congruences
$$uv \equiv x^py^p+\frac{pc_0}{\alpha^{p+1}}
\bigg(\frac{1}{x^{p+1}}+\frac{1}{y^{p+1}}\biggr)x^py^p+
\frac{p}{\alpha^p}\{y^pg(\alpha x)+x^pg(\alpha y)\},$$
$$(u-v)^{p-1} \equiv 
(x-y)^{p(p-1)}+
\frac{pc_0}{\alpha^{p+1}}\frac{(x-y)^{(p-1)^2}}{xy}+
\frac{p}{\alpha^p}(g(\alpha y)-g(\alpha x))(x-y)^{p(p-2)}$$
$-pf(x,y)(x-y)^{p(p-2)}$
and hence
$$uv(u-v)^{p-1} \equiv \mathcal{Z}^p+\frac{pc_0}{\alpha^{p+1}}\mathcal{Z}^{p-1}+
\frac{p}{\alpha^p}x^py^p(x-y)^{p(p-2)}(g(\alpha y)-g(\alpha x))$$
$$+\frac{pc_0}{\alpha^{p+1}}\bigg(\frac{1}{x^{p+1}}+\frac{1}{y^{p+1}}\biggr)\mathcal{Z}^p+
\frac{p}{\alpha^p}\{y^pg(\alpha x)+x^pg(\alpha y)\}(x-y)^{p(p-1)}-ph(x,y).$$
Therefore, the assertion follows from Lemma \ref{y1} and the congruence 
$\alpha^{p^2}\bigl(\frac{u+v}{uv}\bigr) \equiv \alpha^{p^2} \bigl(\frac{x+y}{xy}\bigr)^p$ 
 modulo $p\alpha$ by (\ref{yy1}) and (\ref{yy2}).
\end{proof}

\begin{corollary}\label{good}
Over $R_K:=\mathbb{Z}_p[\alpha_1] \otimes \mathbb{Z}_{p^2},$
the reduction of $\mathbf{Y}^A_{2,2}$ is a reduced, connected, affine curve
of genus $(p-1)/2$
with one only branch through each singular point and with $4$
points at infinity which has the following
equations
$$xy(x-y)^{p-1}=\bar{c}_0,$$
$$Z^p+\bar{c}_0+\bar{c}_0^2\biggl(\frac{1}{x^{p+1}}+\frac{1}{y^{p+1}}\biggr)=0.$$
The curve $\overline{\mathbf{Y}}^A_{2,2}$ has $(p+1)$ singular points at $(x,y)=(-\zeta,\zeta)$
with $\zeta^{p+1}=-\bar{c}_0.$
\end{corollary}
\begin{proof}
We introduce a new parameter $Z$ as follows
\begin{equation}\label{yyy1}
\mathcal{Z}=xy(x-y)^{p-1}=c_0+\gamma Z+\alpha\frac{g(\alpha y)-g(\alpha x)}{x-y}xy
\end{equation}
We set $\phi:=\alpha(g(\alpha y)-g(\alpha x)).$
Since we have $v(\gamma^p)>v(\alpha^{p+2})$,
 we acquire a congruence $\{xy(x-y)^{p-1}\}^p=c_0$
modulo $\alpha^{p+2}$
by (\ref{yyy1}).
Hence, by substituting (\ref{yyy1}) to the congruence in Lemma \ref{yyq1},
and dividing it by $p/\alpha^{p+1},$
we obtain the following congruence
$$Z^p+c_0^2\{xy(x-y)^{p-1}\}^{-1}+
\frac{c_0\phi}{(x-y)^p}+c_0^2\biggl(\frac{1}{x^{p+1}}+\frac{1}{y^{p+1}}\biggr)$$
\begin{equation}\label{yyy2}
+\alpha(y^pg(\alpha x)+x^pg(\alpha y))(x-y)^{p(p-1)}=
c_0c_1\alpha \biggl(\frac{x+y}{xy}\biggr)^p-\alpha^{p+1}H\ ({\rm mod}\ \alpha^{p+2}).
\end{equation}
Hence, the required assertion follows from (\ref{yyy2}) and $xy(x-y)^{p-1}=\bar{c}_0$
by (\ref{yyy1}).
\end{proof}

\begin{remark}
In the same way as in \cite[Corollary 5.4]{CM},
for {\it any} supersingular elliptic curve
$A,$ ``the new bridging component'', which appears in the reduction of the supersingular locus 
$W_A(p^4) \subset X_0(p^4),$
must have the following equations
$$x^2=1+y^{(p+1)/i(A)},Z^p+1+2(x^{p+1}+1)/(x^2-1)^p=0$$
and genus $(p+1)/i(A)-1.$
\end{remark}

\subsection{New components in the stable reduction of $X_0(p^4)$}

In this subsection, we analyze the singular residue classes in the space $\mathbf{Y}^A_{2,2}.$
We prove that $(p+1)$ the Deligne-Lusztig curve for ${\rm SL}_2(\mathbb{F}_p),$
which is defined by $a^p-a=t^{p+1},$
appear in the stable reduction of $W_A(p^4).$
These components attach to the reduction $\overline{\mathbf{Y}}^A_{2,2}$
at $(p+1)$ singular points.
This is a new phenomenon which is observed in the the stable reduction of $X_0(p^4).$ 

We change variables as follows
$$x=\frac{r+1}{2s},y=\frac{r-1}{2s}$$
as in \cite[the proof of Proposition 5.1]{CM}.
Under these variables $(r,s),$ the congruence (\ref{yyy1}) has the following form
\begin{equation}\label{uu1}
\frac{r^2-1}{4s^{p+1}}(1-s^p\phi)=c_0+\gamma Z.
\end{equation}
In the following, we rewrite the terms in the left hand side of the congruence (\ref{yyy2})
under the variables $(r,s).$

\begin{lemma}\label{qq1}
The terms in the left hand side of the congruence (\ref{yyy2}) have 
the following forms modulo $\alpha^{p+2}$ under the variables $(r,s)$
\\1. $ c_0^2\{xy(x-y)^{p-1}\}^{-1}=
c_0^2 \bigl(\frac{1-s^p\phi}{c_0+\gamma Z}\bigr).$ \\
2. $\frac{c_0\phi}{(x-y)^p}=c_0s^p\phi.$ \\
3. $c_0^2\bigl(\frac{1}{x^{p+1}}+
\frac{1}{y^{p+1}}\bigr)=c_0^2 \times 
\frac{2(1+r^{p+1})}{(r^2-1)^p} \cdot \bigl(\frac{1-s^p\phi}{c_0+\gamma Z}\bigr).$ \\
4. $\alpha(y^pg(\alpha x)+x^pg(\alpha y))(x-y)^{p(p-1)}
=\frac{\phi}{2s^{p^2}}+
\frac{\alpha r^p}{2s^{p^2}}\{g(\alpha x)+g(\alpha y)\}.$ \\
\end{lemma}
\begin{proof}
The required congruences 1, 2 follow from (\ref{uu1}) immediately.
We prove the congruence 3.
We have the following equality
$$c_0^2\biggl(\frac{1}{x^{p+1}}+\frac{1}{y^{p+1}}\biggr)=c_0^2(2s)^{p+1}\frac{(r+1)^{p+1}+(r-1)^{p+1}}{(r^2-1)^{p+1}}.$$
Since we have $(r+1)^{p+1}+(r-1)^{p+1} \equiv 2(1+r^{p+1})$ modulo
$\alpha^{p+2},$ the required assertion follows from (\ref{uu1}).

Finally, we prove the congruence 4.
By $x-y=1/s,$
we acquire $\alpha(y^pg(\alpha x)+x^pg(\alpha y))(x-y)^{p(p-1)}
\equiv \bigl(\alpha/2s^{p^2}\bigr) \times [r^p\{g(\alpha x)+g(\alpha y)\}+\{g(\alpha y)-g(\alpha x)\}]\ ({\rm mod}\ \alpha^{p+2}).$
Hence, the required assertion follows from the definition of $\phi.$
\end{proof}

We set $f(r):=\frac{2(1+r^{p+1})}{(r^2-1)^p}.$
By the congruences 1-4 in Lemma \ref{qq1}, the left hand side of the congruence (\ref{yyy2})
is congruent to the following modulo $\alpha^{p+2}$
\begin{equation}\label{zz}
Z^p+\biggl(\frac{c_0^2(1+f(r))}{c_0+\gamma Z}\biggr)
+\biggl(c_0s^p-\frac{c_0^2(1+f(r))s^p}{c_0+\gamma Z}+\frac{1}{2s^{p^2}}\biggr)\phi
+\frac{\alpha r^p}{2s^{p^2}}\{g(\alpha x)+g(\alpha y)\}.
\end{equation}
In the following, we change variables again as follows 
$$r=\alpha t, s=s_0+\alpha^2s_1$$
where $s_0$ satisfies $4c_0s_0^{p+1}+1=0\ ({\rm mod}\ \alpha^{p+2}).$
\begin{lemma}
Let the notation be as above.
Further, let $c_2$ be the leading coefficient of $(g(X)-c_1)/X.$
Then, the parameter $s_1$ is written with respect to $t$ as follows
$$s_1=\biggl(t^2-\frac{c_2}{s_0}\biggr)/s_0^p\ ({\rm mod}\ \alpha).$$
\end{lemma}
\begin{proof}
Substituting $r=\alpha t$ and $s=s_0+\alpha^2 s_1$
to (\ref{uu1}) and considering it modulo $\alpha^{3},$
we acquire the following congruence, by $\phi \equiv -c_2\alpha^2/s_0\ ({\rm mod}\ \alpha^3)$
\begin{equation}\label{oo1}
(4c_0s_0^{p+1}+1)+\alpha^2s_0^ps_1 \equiv \alpha^2\biggl(t^2-\frac{c_2}{s_0}\biggr)\ ({\rm mod}\ \alpha^3).
\end{equation}
Hence, the required assertion follows from the definition of $s_0.$
\end{proof}
\begin{lemma}\label{qqq1}
The congruence (\ref{yyy2}) modulo $\alpha^2\gamma$ has the following form under the variables $(Z,t)$
$$Z^p-c_0+\gamma Z-2c_0(\alpha t)^{p+1}+d=0\ ({\rm mod}\ \alpha^2 \gamma)$$
where $d=\alpha^{p+1}H(1/2s_0,-1/2s_0,(1/2s_0)^p,-(1/2s_0)^p).$
\end{lemma}

\begin{proof}First, note that $v(\alpha^2 \gamma)=2/p(p+1)+(p-1)/p^2<v(\alpha^{p+2}).$
Since we have $(r^2-1)^p=-1+(\alpha t)^{2p}=-1$
modulo $\alpha^{p+2},$ the following congruence holds $f(r) \equiv -2(1+(\alpha t)^{p+1})$
mod $\alpha^{p+2}.$ Note that $s^p \equiv s_0^p$ modulo $\alpha^{p+2}.$ Thereby, the term in (\ref{zz})
$$\biggl(c_0s^p-\frac{c_0^2(1+f(r))s^p}{c_0+ \gamma Z}+\frac{1}{2s^{p^2}}\biggr)\phi$$
 is congruent to the following
$$\frac{c_0(4c_0s_0^{p(p+1)}+1)+(2c_0s^{p(p+1)}+1)\gamma Z}{2s_0^{p^2}(c_0+ \gamma Z)}\phi+\frac{2c_0^2s_0^p(\alpha t)^{p+1}}{c_0+\gamma Z}\phi$$
modulo $\alpha^{p+2}.$
Since $\phi$ is divisible by $\alpha^2,$ the above term is zero modulo $\alpha^2 \gamma$ by the definition of $s_0.$
The term $\frac{\alpha r^p}{2s^{p^2}}\{g(\alpha x)+g(\alpha y)\}$ in
 (\ref{zz}) 
has the following form $c_1\alpha^{p+1}t^p/s_0^{p^2}$ modulo $\alpha^2 \gamma.$
The term $c_0c_1\alpha\bigl(\frac{x+y}{xy}\bigr)^p$
in the right hand side of (\ref{yyy2}) has the following form $-4c_0c_1\alpha^{p+1}(s_0t)^p$
modulo $\alpha^2 \gamma.$
Hence, we know that the congruence (\ref{yyy2}) has the following form, by (\ref{zz}) and
the above argument, 
$$Z^p-\frac{c_0^2}{c_0+\gamma Z}-\frac{2c_0^2(\alpha t)^{p+1}}{c_0+\gamma Z}+
\biggl(\frac{4c_0s_0^{p(p+1)}+1}{s_0^{p^2}}\biggr)c_1\alpha^{p+1}t^p+d=0\ ({\rm mod}\ \alpha^2 \gamma).$$
Since we have $v(\gamma^2)>v(\alpha^2 \gamma)$
and hence $(c_0+\gamma Z)^{-1}=1/c_0-\gamma Z/c_0^2$
modulo $\alpha^2 \gamma,$ the above congruence induces the required congruence.
\end{proof}
We choose a root $\gamma_0$ such that $Z^p+\gamma Z-c_0+d=0\ ({\rm mod}\ \alpha^2 \gamma)$
and an element $\beta$ such that $\beta^{p-1}=-\gamma.$
We have $v(\beta)=1/p^2$ and $v(\beta^p)=v(\alpha^{p+1}).$

\begin{corollary}\label{op}
In the stable reduction of $X_0(p^4)$, for each supersingular elliptic curve $A/\mathbb{F}_p$
with $j(A) \neq 0,1728,$ there exist $(p+1)$ irreducible components, 
which attach to $\overline{\mathbf{Y}}^A_{2,2}$ 
at its singular points
and is a reduced, connected,
affine curve $p(p-1)/2$ defined by the following equation
$$a^p-a=\bar{c}t^{p+1}.$$
\end{corollary}
\begin{proof}
We change a variable $Z=\gamma_0+\beta a.$
Then, the congruence in Lemma \ref{qqq1} has the following
form by the definition of $\beta$
$$\beta^p(a^p-a)-2c_0(\alpha t)^{p+1}=0\ ({\rm mod}\ \alpha^2\gamma).$$
By dividing this by $\beta^p$ and putting $c:=2c_0\bigl(\frac{\alpha^{p+1}}{\beta^p}\bigr),$
we acquire the following congruence 
$$a^p-a=ct^{p+1}\ ({\rm mod}\ \frac{\alpha^2 \gamma}{\beta^p}).$$
Note that $c$ is a unit.
Hence, we obtain the Deligne-Lusztig curve $a^p-a=\bar{c}t^{p+1}.$
\end{proof}
\begin{remark}
We obtained  the curves $a^p-a=t^{p+1}$  in the reduction of $W_A(p^4)$ in Corollary \ref{op}.
The equation $a^p-a=t^{p+1}$ has the following form, by changing variables $a=X/Y,t=1/Y,$
$$X^pY-XY^p=1.$$
The affine curve defined by $X^pY-XY^p=1$ is called the Deligne-Lusztig curve for ${\rm SL}_2(\mathbb{F}_p).$
Y.\ Mieda pointed out this fact to the author.
See \cite[subsection 1.3]{W} for the Deligne-Lusztig curve.
\end{remark}
\begin{remark}
Let $T$ denote the singular residue class of $\mathbf{Y}^A_{2,2}.$
We put $s=s_0+u,Z=\gamma_0+z.$
By (\ref{oo1}) and Lemma \ref{qqq1},
the space $\mathbf{Y}^A_{2,2}$ is defined by 
 $$s_0^pu=r^2-c_2\alpha^2/s_0+\alpha^3h(r,u,z), z^p+\gamma z-2c_0r^{p+1}=\alpha^2\gamma g(r,u,z)$$
 for some rigid analytic functions $h(r,u,z), g(r,u,z).$
Hence, the space $T$ and its underlying affinoid $\mathbf{X}_T$
have the following descriptions 
$$T(\mathbb{C}_p)=\{(r,u,z) \in \mathbb{C}_p^{3}|v(r)>0,v(z)>0,v(u)>0,
s_0^pu=r^2-c_2\alpha^2/s_0+\alpha^3h(r,u,z),$$
$$z^p+\gamma z-2c_0r^{p+1}=\alpha^2 \gamma g(r,u,z)
\},$$
and $\mathbf{X}_T(\mathbb{C}_p)=\{(r,u,z) 
\in T(\mathbb{C}_p)|v(r) \geq v(\alpha),v(z) \geq v(\beta),v(u) \geq v(\alpha^2)\}.$
Therefore, the complement $T \backslash \mathbf{X}_T$
is isomorphic to an annulus $A(p^{-1/p^2(p+1)},1)$ by the following maps
$z=x^{p+1}+ \cdots, u=x^{2p}+\cdots, r=x^p+\cdots.$
Hence, we obtain the width $1/p^2(p+1)$ of the annulus $T \backslash \mathbf{X}_T.$
\end{remark}

\section{Stable reduction of $X_0(p^4)$}
In this section, we give the stable covering of $X_0(p^4).$
Namely, we give a covering by {\it basic}
wide opens, whose intersections 
are annuli.
Some of these wide opens, namely $W^{\pm}_{a,b}\ 
(a+b=4,a \geq 0,b \geq 0)$,
have already been defined in \cite[Section 3]{CM}.
The spaces $W_{a,b}\ (a+b=4, a\geq 0, b\geq 0)$
cover the ordinary locus of $X_0(p^4)$.
From our analysis of $\mathbf{Y}^A_{2,2}$
in the previous section, we know that $W_A(p^4)$
is not a basic wide open.
Our next task is to specify some new wide open subspaces
which cover each $W_A(p^4)$ and which can ultimately be shown to be basic.
In subsection 5.1, we review the stable covering of $X_0(p^3)$
from \cite[Section 9]{CM}.
In subsection 5.2, we construct a covering of $X_0(p^4)$
and show that the covering is stable in Theorem \ref{otto}.
In subsection 5.3, we will give a genus computation as in \cite[subsection 5.3]{Mc}.
In subsection 5.4, we will also give intersection multiplicity 
data for $X_0(p^n)\ (n=2,3,4).$
\subsection{Review of the stable covering of $X_0(p^3)$ in \cite[Section 9]{CM}}
In this subsection, we briefly recall the construction of the stable covering of the 
modular curve $X_0(p^3)$ from \cite[Section 9]{CM} for the convenience of a reader.
Now suppose that $A$ is any supersingular elliptic curve mod $p.$
Identify $W_A(p)$
with an annulus $A(p^{-i(A)},1)$
as explained in subsection 2.2. 
\begin{lemma}
Let the notation be as in subsections 2.2 and 2.3.
We have the following inclusions
$$\pi_{1,1}(\mathbf{Y}^A_{2,1}) \subset \mathbf{C}^A_{i(A)/(p+1)},
\pi_{1,1}(\mathbf{Y}^A_{1,2}) \subset \mathbf{C}^A_{pi(A)/(p+1)},
\pi_{1,1}(\mathbf{Z}^A_{1,1}) \subset \mathbf{C}^A_{i(A)/2}.$$
\end{lemma}
\begin{proof}
This follows from Lemma \ref{le}.
\end{proof}
Now, we recall the construction of the stable covering of $X_0(p^3)$
from loc.\ cit.
We define as follows
$$\mathbf{Y}^A_{2,1} \subset V_2(A):=\pi_{1,1}^{-1}A(p^{-i(A)/2},1),$$
$$\mathbf{Z}^A_{1,1} \subset U(A):=\pi_{1,1}^{-1}A(p^{-pi(A)/(p+1)},p^{-i(A)/(p+1)}),$$
$$\mathbf{Y}^A_{1,2} \subset V_1(A):=\pi_{1,1}^{-1}A(p^{-i(A)},p^{-i(A)/2}).$$

Let $\mathcal{S}(A)$ be the set of the singular residue classes of $\mathbf{Z}^A_{1,1}$
and, for $S \in \mathcal{S}(A),$
let $\mathbf{X}_S$
be the underlying affinoid subdomain of $S.$
Let $\hat{U}(A):=U(A) \backslash \bigcup_{S \in \mathcal{S}(A)}\mathbf{X}_S.$

Let $(E,C) \in X_0(p^n).$ Let $a \geq b \geq 0$ be integers such that $a+b=n.$
Assume that $|K(E)| \geq p^n$ and $|K(E) \cap C|=p^a.$
We briefly recall the definition of the pairing on the canonical subgroup $K_a(E)$
in \cite[subsection 3.2]{CM}.
Let $A,B \in K_a(E).$
Then, we can choose $P \in C$
and $Q \in K_n(E)$
such that $p^bP=A$ and $p^bQ=B.$
Now, set $\mathcal{P}_{E,C}(A,B):=e_n(P,Q)$
where $e_n(-,-)$ denote the Weil pairing on $E[p^n]$.
This gives a well-defined pairing of $K_a(E)$
with itself onto 
$\mu_{p^b}.$
See \cite[(2.8.41) and (2.8.7)]{KM}
to check the well-definedness of the pairing.
Further, there are exactly two isomorphism classes of the pairing on 
$\mathbb{Z}/p^a\mathbb{Z}$
onto $\mu_{p^b}.$
Let $e^{\pm}$ be representatives for these classes.

We recall the spaces $W_{a,b} \subset X_0(p^3)$
for $a+b=3,a \geq 0,b \geq 0,$
which cover the ordinary locus, from loc.\ cit.
For $a+b=3,a>b,$
we set 
$$W_{a,b}:=\{(E,C) \in X_0(p^3)|\ |K(E)| \geq p^3, |K(E) \cap C|=p^a\}$$
and for $a <b,$
put $W_{a,b}:=w_3(W_{b,a}).$
We define
$$W^{\pm}_{a,b}=\{(E,C) \in W_{a,b}|\ (K_a(E),\mathcal{P}_{E,C}) \simeq (\mathbb{Z}/p^a\mathbb{Z},e^{\pm})\}$$
Set $W^{\pm}_{n,0}:=W_{n,0}$
and for $b \geq a \geq 0,$ we set
$$W^{\beta}_{a,b}:=w_3(W^{(\frac{-1}{p})\beta}_{b,a}).$$
Then, the irreducible affnoids $\mathbf{X}^{\pm}_{a,b}$
defined in \cite{C2} are affinoids whose points are those $(E,C) \in W^{\pm}_{a,b}$
for which $E$ has ordinary or multiplicative reduction.
The affinoid $\mathbf{X}^{\pm}_{a,b}$
reduces to the Igusa curve ${\rm Ig}(p^{{\rm min}(a,b)})$.
See \cite[subsection 3.2]{CM} or \cite[subsection 2.1]{CM2}
for more detail.

We define a covering of the modular curve of $X_0(p^3)$
as follows as in \cite[Section 9]{CM}
$$\mathcal{C}_0(p^3):=
\{\{V_1(A),V_2(A),\hat{U}(A)\} 
\cup \mathcal{S}(A)\}_{A:{\rm supersingular}} \cup 
\{W_{3,0},W_{0,3},W^{\pm}_{2,1},W^{\pm}_{1,2}\}.$$
In \cite[the proof of Theorem 9.2]{CM}, 
the following three lemmas are proved to show that this covering is actually stable.
In other words, in loc.\ cit., it is proved that all intersections of the wide open spaces 
of the covering $\mathcal{C}_0(p^3)$
are annuli.
See \cite[subsection 2.2]{CM}
for the notion of the stable covering.
We put $U_{a,b}(A):=W_{a,b} \cap W_A(p^3)$
and $U^{\pm}_{a,b}(A):=W^{\pm}_{a,b} \cap W_A(p^3)$ 
respectively.
\begin{lemma}(\cite[Theorem 9.2]{CM})\label{lee1}
Let the notation be as above.
Then, the spaces $U_{3,0}(A)$ and $U_{0,3}(A)$
are annuli of width $\frac{i(A)}{p(p+1)}.$
\end{lemma}
\begin{proof}
We prove the assertion in the same way as in \cite[the proof of Theorem 9.2]{CM}.
It suffices to show the required assertion for $U_{3,0}(A),$
since $U_{0,3}(A)$ is isomorphic to $U_{3,0}(A)$
by the Atkin-Lehner involution. Let $(E,C) \in U_{3,0}(A)$.
Then, we have $K_3(E)=C$ and $0 <h(E)<1/p(p+1).$
Therefore, the map $\pi_{0,2}$ induces an  isomorphism from $U_{3,0}(A)$
to an annulus $A(p^{-i(A)/p(p+1)},1)$.
Hence, the required assertion follows.
\end{proof}

\begin{lemma}(\cite[Theorem 9.2]{CM})\label{lee2}
Let the notation be as above.
For $(a,b)=(2,1),(1,2),$ the space $U_{a,b}(A)$
has two connected components $U^{\pm}_{a,b}(A).$
Furthermore, $U^{\pm}_{a,b}(A)$
are annuli of width $\frac{2i(A)}{p(p^2-1)}.$
\end{lemma}
\begin{proof}
We only prove the assertions for $U_{2,1}(A).$
Let $(E,C) \in U_{2,1}(A).$
Note that $pC=K_2(E)$
and $0 <h(E)<1/p(p+1).$
For simplicity, we prove the assertions for case $i(A)=1.$
Similarly as Corollary \ref{prod},
the space $U_{2,1}(A)$ is identified with the following space,
 by the embedding $\prod_{0 \leq i \leq 3} \pi_{i,3-i},$
$$\{(X_i)_{i=0}^{i=3} \in W_A(1)^{ \times 4}|\ 0< v(X_0)<\frac{1}{p(p+1)}, v(X_1)=pv(X_0),
v(X_2)=p^2v(X_0), v(X_3)=pv(X_0),$$
$X_1 \neq X_3, F_p^{\beta_0}(X_i,X_{i+1})=0\ (0 \leq i \leq 2)\}.$
Thereby, the equations $F_p^{\beta_0}(X_i,X_{i+1})=0\ (0 \leq i \leq 2)$
induce the followings
$$X_1=X_0^p+\frac{pc_0}{X_0}+pg(X_0), 
X_2=X_1^p+\frac{pc_0}{X_1}+pg(X_1)=
X_3^p+\frac{pc_0}{X_3}+pg(X_3)$$
where we have $v(g(X_i))>0\ (i=1,2,3).$
Hence, by $X_1 \neq X_3,$ we acquire the following equality 
\begin{equation}\label{ty1}
X_1X_3(X_1-X_3)^{p-1}=pc_0+{\rm higher}\ {\rm terms}.
\end{equation}
We set $Z:=X_1-X_3.$
Then, by (\ref{ty1}) and $v(X_1) < 1/(p+1),$ we acquire $v(X_1^2Z^{p-1})=1$
and $v(X_1)<v(Z).$
Thereby, we obtain the following by (\ref{ty1})
$$X_1^2Z^{p-1}=pc_0+{\rm higher}\ {\rm terms}$$
where the valuation of the higher terms is bigger than $1.$
This is decomposed to the following two equalities
\begin{equation}\label{ty2} 
X_1Z^{(p-1)/2}={\pm}(pc_0)^{1/2}+{\rm higher}\ {\rm terms}
\end{equation}
where the valuation of the higher terms is bigger than $1/2.$
Substituting $X_1=X_0^p+(pc_0/X_0)+pg(X_0)$
to this equality, the following hols
$$X_0^pZ^{(p-1)/2}={\pm}(pc_0)^{1/2}+{\rm higher}\ {\rm terms}.$$
Therefore, we can write $X_0=t^{(p-1)/2}+{\rm higher}\ {\rm terms}$
and
$Z={\pm}(pc_0)^{1/(p-1)}/t^p+{\rm higher}\ {\rm terms}$
with $0<v(t)<2/p(p^2-1).$
Hence, the required assertions follow from $(p,(p-1)/2)=1.$
\end{proof}
\begin{remark}
The spaces $U^{\pm}_{2,1}(A)$ and $U_{1,2}^{\pm}(A)$ map onto $B:=A(p^{-\frac{i(A)}{p(p+1)}},1)$
via $\pi_{0,2}$ and $\pi_{2,0}$ respectively,
but each with degree $(p-1)/2$
as proved in \cite[the proof of Theorem 9.2]{CM}. Then, we know that
the spaces $U_{2,1}^{\pm}(A), U_{1,2}^{\pm}(A)$
are annuli of width $2i(A)/p(p^2-1)$ by \cite[Lemma 2.1]{CM}.
\end{remark}
\begin{lemma}(\cite[Theorem 9.2]{CM})\label{lee3}
Let the notation be as above.
\\1.\ We have the followings
$$V_2(A) \cap U(A)=\{(E,C) \in W_A(p^3)|\frac{1}{p(p+1)}<h(E)<\frac{1}{2p},|K(E) \cap C|=p^2\},$$
$$V_1(A) \cap U(A)=\{(E,C) \in W_A(p^3)|\frac{1}{2p}<h(E)<\frac{1}{p+1},|K(E) \cap C|=p^2\}.$$
\\2.\ The intersections $V_2(A) \cap U(A)$ and $V_1(A) \cap U(A)$ are annuli of width $\frac{(p-1)i(A)}{2p^2(p+1)}$
\end{lemma}
\begin{proof}
Let $(E,C) \in V_2(A) \cap U(A).$
By the definitions of $V_2(A)$ and $U(A),$ we have $\pi_{1,1}(E,C) \in A(p^{-i(A)/2},p^{-i(A)/(p+1)}).$
By using Theorem \ref{Ka}, we obtain the assertion 1 for $V_2(A) \cap U(A).$
For $V_1(A) \cap U(A)$, we prove 
the required assertion
in the same way as above. We omit the detail. 

We prove the assertion 2.
We prove the assertion only for $V_2(A) \cap U(A),$
because the argument is very similar.
For simplicity, we assume $i(A)=1.$
By considering 
an embedding $\prod_{0 \leq i \leq 3} \pi_{i,3-i}$ as in Corollary \ref{prod},
we acquire the following description of $V_2(A) \cap U(A)$
$$\{(X_i)_{i=0}^{i=4} \in W_A(1)^{\times 4}|\frac{1}{p(p+1)}<v(X_0)<\frac{1}{2p}, v(X_1)=pv(X_0),
v(X_1)+v(X_2)=1,$$
$$v(X_3)=v(X_2)/p, F_p^{\beta_0}(X_i,X_{i+1})=0\ (0 \leq i \leq 2)\}.$$
Then, we have
$1/(p+1)<v(X_1)<1/2<v(X_2)<p/(p+1).$
The equations $F_{p}^{\beta_0}(X_i,X_{i+1})=0\ (0 \leq i \leq 2)$ induce the following equalities 
$$X_1=X_0^p+\frac{pc_0}{X_0}+pg(X_0), X_2=X_1^p+\frac{pc_0}{X_1}+pg(X_1)=X_3^p+\frac{pc_0}{X_3}+pg(X_3).$$
We rewrite $X_2=X_1^p+(pc_0/X_1)+pg(X_1)$ as follows
\begin{equation}\label{bv}
X_1X_2=pc_0+X_1^{p+1}+pX_1g(X_1).
\end{equation}
Note that $v(X_0)+v(X_3)=1/p, 1/p(p+1)<v(X_0)<1/2p<v(X_3)<1/(p+1).$
Substituting $X_1=X_0^p+(pc_0/X_0)+pg(X_0)$ and $X_2=X_3^p+(pc_0/X_3)+pg(X_3)$ to 
the equality (\ref{bv}), we obtain
the following
\begin{equation}\label{gi}
X_0^pX_3^p+pc_0\frac{X_0^p}{X_3}
=pc_0+{\rm higher}\ {\rm terms}
\end{equation}
where the valuation of the higher terms in the above equality 
is bigger than $((p-1)/p)+(p+1)v(X_0).$
We choose a root $\gamma$ such that $\gamma^p=pc_0.$
We put $X_0X_3=\gamma+Z$
with $v(Z)>v(\gamma)=1/p.$
By substituting $X_0X_3=\gamma+Z$ to (\ref{gi}), we acquire
the following
\begin{equation}\label{st3}
Z^p+pc_0X_0^{p+1}/\gamma
={\rm higher}\ {\rm terms}
\end{equation}
where the valuation of the higher terms
is bigger than $v(Z^p).$
Hence, we acquire $v(Z)=\bigl((p-1)/p^2\bigr)+\bigl((p+1)v(X_0)/p\bigr).$
Since the condition $v(Z)>v(\gamma)$
is equivalent to $v(X_0)>1/p(p+1),$ we acquire $v(Z)>1/p.$
We choose a root $\gamma_1$ satisfying $\gamma_1^p=pc_0/\gamma.$
Note that $v(Z)>v(\gamma_1)+v(X_0).$
We change a variable as follows
$Z_1:=-Z/\gamma_1X_0.$
Then, the equality (\ref{st3}) is written as follows
$$X_0=Z_1^p+{\rm higher}\ {\rm terms}$$
with $1/p^2(p+1)<v(Z_1)<1/2p^2.$
Hence, the required assertion follows.
\end{proof}

\subsection{Construction of the stable covering of $X_0(p^4)$}
In this subsection, we will give the stable covering of $X_0(p^4)$
similarly as in \cite[Section 9]{CM}.
Now, suppose that $A$ is any supersingular elliptic curve mod $p.$
Identify $W_A(p)$
with $A(p^{-i(A)},1)$ as explained in subsection 2.2.
\begin{lemma} Let the notation be as in subsections 2.2 and 2.3.
\\1. We have the following inclusions
$$\pi_{1,2}(\mathbf{Y}^A_{3,1}) \subset \mathbf{C}^A_{{i(A)/p(p+1)}},
\pi_{1,2}(\mathbf{Y}^A_{2,2}) \subset \mathbf{C}^A_{i(A)/(p+1)},
\pi_{1,2}(\mathbf{Y}^A_{1,3}) \subset \mathbf{C}_{i(A)p/(p+1)},$$ and
$$\pi_{1,2}(\mathbf{Z}^A_{2,1}) \subset \mathbf{C}_{i(A)/2p},
\pi_{1,2}(\mathbf{Z}^A_{1,2}) \subset \mathbf{C}^A_{i(A)/2}.$$
\\2. We have the following inclusions
$$\pi_{2,1}(\mathbf{Y}^A_{3,1}) \subset \mathbf{C}^A_{i(A)/(p+1)},
\pi_{2,1}(\mathbf{Y}^A_{2,2}) \subset \mathbf{C}^A_{i(A)p/(p+1)},
\pi_{2,1}(\mathbf{Y}^A_{1,3}) \subset \mathbf{C}^A_{i(A)(1-/p(p+1))},$$
and
$$\pi_{2,1}(\mathbf{Z}^A_{2,1}) \subset \mathbf{C}^A_{i(A)/2},
\pi_{2,1}(\mathbf{Z}^A_{1,2}) \subset \mathbf{C}^A_{i(A)(1-1/2p)}.$$
\end{lemma}
\begin{proof}
The assertions follow immediately from Lemma \ref{key}.
We omit the detail.
\end{proof}
Then, we define five subspaces of $W_A(p^4)$ as follows
$$\mathbf{Y}^A_{3,1} \subset V_3(A):=\pi_{1,2}^{-1}A(p^{-i(A)/2p},1) \cap \pi_{2,1}^{-1}A(p^{-i(A)/2},1),$$
$$\mathbf{Y}^A_{2,2} \subset V_2(A):=\pi_{1,2}^{-1}A(p^{-i(A)/2},1) \cap \pi_{2,1}^{-1}A(p^{-i(A)},p^{-i(A)/2}),$$
$$\mathbf{Y}^A_{1,3} \subset V_1(A):=\pi_{1,2}^{-1}A(p^{-i(A)},p^{-i(A)/2}) \cap \pi_{2,1}^{-1}A(p^{-i(A)},p^{-(1-(1/2p))i(A)}),$$
$$\mathbf{Z}^A_{2,1} \subset U_2(A):=\pi_{1,2}^{-1}A(p^{-i(A)/(p+1)},p^{-i(A)/p(p+1)}) \cap 
\pi_{2,1}^{-1}A(p^{-pi(A)/(p+1)},p^{-i(A)/(p+1)}),$$
$$\mathbf{Z}^A_{1,2} \subset U_1(A):=\pi_{1,2}^{-1}A(p^{-pi(A)/(p+1)},p^{-i(A)/(p+1)}) \cap \pi_{2,1}^{-1}
A(p^{-i(A)(1-1/p(p+1))},p^{-pi(A)/(p+1)}).$$
Let $\mathcal{S}_1(A)$ and $\mathcal{S}_2(A)$
 denote the sets of singular residue classes of $\mathbf{Z}^A_{1,2}$
and of $\mathbf{Z}^A_{2,1}$ respectively
and, for each $S \in \mathcal{S}_i(A)\ (i=1,2),$
let $\mathbf{X}_S$ be the underlying affinoid of $S$
as in \cite[Section 9]{CM}.
Let $\hat{U}_i(A)\ (i=1,2)$ denote the wide open given by 
$$\hat{U}_i(A):=U_i(A) \backslash \bigcup_{S \in \mathcal{S}_i(A)}\mathbf{X}_S.$$
In the same way, we define $\mathcal{T}(A)$
to be the set of the singular residue classes of $\mathbf{Y}^A_{2,2}$
and, for each $T \in \mathcal{T}(A),$
let $\mathbf{X}_T$ be the underlying affinoid of $T.$
Let $\hat{V}_2(A)$ denote the wide open given by 
$$\hat{V}_2(A):=V_2(A) \backslash \bigcup_{T \in \mathcal{T}(A)}\mathbf{X}_T.$$

We define the spaces 
 $W_{a,b}\ (a+b=4, a \geq 0, b \geq 0)$,
 which cover the ordinary locus.
For $(a,b)=(4,0),(3,1),$ we put
$$W_{a,b}:=\{(E,C) \in X_0(p^4)|\ |K(E)|\geq p^4, |K(E) \cap C|=p^a\}.$$
For $(a,b)=(0,4),(1,3),$ we set $W_{a,b}:=w_{4}(W_{b,a}).$
Further, we define
$$W_{2,2}:=\{(E,C) \in X_0(p^4)|\ |K(E)|\geq p^3,|K(E) \cap C|=p^2\}.$$
We define as follows $W^{\pm}_{4,0}:=W_{4,0}, W^{\pm}_{0,4}:=W_{0,4},$
$$W^{\pm}_{3,1}:=\{(E,C) \in W_{3,1}|(K_3(E),\mathcal{P}_{E,C}) \simeq (\mathbb{Z}/p^3\mathbb{Z},e^{\pm})\},$$
$$W^{\pm}_{2,2}:=\{(E,C) \in W_{2,2}|(K_2(E),\mathcal{P}_{E,C}) \simeq (\mathbb{Z}/p^2\mathbb{Z},e^{\pm})\}.$$
See the previous subsection for the pairing $e^{\pm}$ on the canonical subgroups.
For $a<b,$ we set $W^{\pm}_{a,b}=w_4(W^{\pm}_{b,a}).$

\begin{lemma}
We put $U_{a,b}:=W_{a,b} \cap W_A(p^4).$
Then, the space $W_A(p^4)$
is the union of the following
spaces
$\{V_i(A)\}_{1 \leq i \leq 3},\{U_j(A)\}_{j=1,2}$
and $\{U_{a,b}\}_{a+b=4,a\geq 0, b \geq 0}.$
\end{lemma}
\begin{proof}
The assertion follows from Theorem \ref{Ka}. We omit the detail.
\end{proof}

We define a covering $\mathcal{C}_0(p^4)$ of $X_0(p^4)$
as the union of the following sets
$$\{W^{\pm}_{a,b}|a,b \geq 0, a+b=4\}$$
and over all supersingular curves $A$ of
$$\{V_1(A),V_3(A),\hat{V}_2(A),\hat{U}_1(A),\hat{U}_2(A)\} \cup
\mathcal{S}_1(A) \cup \mathcal{S}_2(A) \cup \mathcal{T}(A).$$
We put $U^{\pm}_{a,b}(A):=W^{\pm}_{a,b} \cap W_A(p^4).$
In the following, we will prove that all intersections of
the elements of the covering $\mathcal{C}_0(p^4)$
are annuli.
\begin{lemma}\label{Lee1}
Let the notation be as above.
Then, the spaces $U_{4,0}(A)$ and $U_{0,4}(A)$
are annuli of width $\frac{i(A)}{p^2(p+1)}.$
\end{lemma}
\begin{proof}
We prove the assertion in the same way as in Lemma \ref{lee1}.
It suffices to show the required assertion for $U_{4,0}(A),$
since $U_{0,4}(A)$ is isomorphic to $U_{4,0}(A)$
by the Atkin-Lehner involution. Let $(E,C) \in U_{4,0}(A)$.
Then, we have $K_4(E)=C$ and $0 <h(E)<1/p^2(p+1).$
Therefore, the map $\pi_{0,3}$ induces an  isomorphism from $U_{4,0}(A)$
to an annulus $A(p^{-i(A)/p^2(p+1)},1)$.
Hence, the required assertion follows.
\end{proof}

\begin{lemma}\label{Lee2}
Let the notation be as above.
For $(a,b)=(3,1),(1,3),$ the space $U_{a,b}(A)$
has two connected components $U^{\pm}_{a,b}(A).$
Furthermore, the spaces $U^{\pm}_{a,b}(A)$
are annuli of width $\frac{2i(A)}{p^2(p^2-1)}.$
\end{lemma}
\begin{proof}
We only prove the assertions for $U_{3,1}(A).$
Let $(E,C) \in U_{3,1}(A).$
Note that $pC=K_3(E)$
and $0 <h(E)<1/p^2(p+1).$
For simplicity, we prove the assertions for the case $i(A)=1.$
Similarly as Corollary \ref{prod},
the space $U_{3,1}(A)$ is identified with the following space,
 by the embedding $\prod_{0 \leq i \leq 4} \pi_{i,4-i},$
$$\{(X_i)_{i=0}^{i=4} \in W_A(1)^{ \times 5}|\ 0< v(X_0)<\frac{1}{p^2(p+1)}, v(X_1)=pv(X_0),
v(X_2)=p^2v(X_0), v(X_3)=p^3v(X_0),$$
$v(X_4)=p^2v(X_0), X_2 \neq X_4, F_p^{\beta_0}(X_i,X_{i+1})=0\ (0 \leq i \leq 3)\}.$
Thereby, the equations $F_p^{\beta_0}(X_i,X_{i+1})=0\ (0 \leq i \leq 3)$
induce the followings
$$X_1=X_0^p+\frac{pc_0}{X_0}+pg(X_0), X_2=X_1^p+\frac{pc_0}{X_1}+pg(X_1),$$
$$X_3=X_2^p+\frac{pc_0}{X_2}+pg(X_2), X_3=X_4^p+\frac{pc_0}{X_4}+pg(X_4).$$
In the same way as (\ref{ty1}), by putting $Z:=X_2-X_4,$ we acquire the following
\begin{equation}\label{pmk}
X_2Z^{(p-1)/2}={\pm}(pc_0)^{1/2}+{\rm higher}\ {\rm terms}
\end{equation}
where the valuation of the higher terms is bigger than $1/2.$
Hence, by $v(X_2)<1/(p+1),$ we acquire $v(X_2Z^{(p-1)/2})=1/2$ and $v(X_2)<v(Z).$
Substituting $X_1=X_0^p+(pc_0/X_0)+pg(X_0), X_2=X_1^p+(pc_0/X_1)+pg(X_1)$
to the equality (\ref{pmk}), the following holds
$$X_0^{p^2}Z^{(p-1)/2}={\pm}(pc_0)^{1/2}+{\rm higher}\ {\rm terms}.$$
Therefore, we can write $X_0=t^{(p-1)/2}+{\rm higher}\ {\rm terms}$
and
$Z={\pm}(pc_0)^{1/(p-1)}/t^{p^2}+{\rm higher}\ {\rm terms}$
with $0<v(t)<2/p^2(p^2-1).$
Hence, the required assertions follow from $(p,(p-1)/2)=1.$
\end{proof}

\begin{remark}
Lemma \ref{Lee2} can be proved in the same way as in \cite[the proof of Theorem 9.2]{CM}.
The spaces $U^{\pm}_{3,1}(A)$ and $U_{1,3}^{\pm}(A)$ map onto $B:=A(p^{-\frac{i(A)}{p^2(p+1)}},1)$
via $\pi_{0,3}$ and $\pi_{3,0}$ respectively,
but each with degree $(p-1)/2$
in the same way as in loc.\ cit.
This gives another proof of Lemma \ref{Lee2}. 
\end{remark}

\begin{proposition}\label{nan1}
Let the notation be as above.
Then, the space $U_{2,2}$ is the union of two connected components
$U_{2,2}^{+}$ and $U_{2,2}^{-}.$
Furthermore, the spaces $U_{2,2}^{+}$ and $U_{2,2}^{-}$
are annuli
of width $\frac{2i(A)}{p^2(p^2-1)}.$
\end{proposition}
\begin{proof}
First, recall that, if $(E,C) \in U_{2,2},$
we have $0<h(E)<1/p(p+1)$ and $K_2(E)=p^2C.$
Then, similarly as Corollary \ref{prod},
the space $U_{2,2}$ is isomorphic to the following space,
 by the embedding $\prod_{0 \leq i \leq 4} \pi_{i,4-i},$
$$\{(X_i)_{i=0}^{i=4} \in W_A(1)^{\times 5}|\ 0<v(X_0)<\frac{1}{p(p+1)},v(X_1)=pv(X_0),
v(X_2)=p^2v(X_0),v(X_3)=pv(X_0),$$
$v(X_4)=v(X_0), X_1 \neq X_3,F_p^{\beta_0}(X_i,X_{i+1})=0\ (0 \leq i \leq 3)\}.$
We have the followings
$$X_1=X_0^p+\frac{pc_0}{X_0}+pg(X_0), X_2=X_1^p+\frac{pc_0}{X_1}+pg(X_1),$$
$$X_2=X_3^p+\frac{pc_0}{X_3}+pg(X_3), X_3=X_4^p+\frac{pc_0}{X_4}+pg(X_4).$$
We put $X^p-Y^p=(X-Y)^p+ph(X,Y).$
Then, the equality
$X_1^p+(pc_0/X_1)+pg(X_1)=X_3^p+(pc_0/X_3)+pg(X_3)$ induces the following equality
\begin{equation}\label{p1}
X_1X_3(X_1-X_3)^{p-1}=p\biggl(c_0+\frac{g(X_3)-g(X_1)}{X_1-X_3}X_1X_3-\frac{h(X_1,X_3)}{X_1-X_3}X_1X_3\biggr).
\end{equation}
We put
$\mathcal{Z}:=X_0X_4(X_0-X_4)^{p-1}$ and
$\psi:=\frac{g(X_4)-g(X_0)}{X_0-X_4}X_0X_4-\frac{h(X_0,X_4)}{X_0-X_4}X_0X_4.$
Substituting $X_1=X_0^p+(pc_0/X_0)+pg(X_0), X_3=X_4^p+(pc_0/X_4)+pg(X_4)$
to (\ref{p1}),
we obtain
\begin{equation}\label{err}
\mathcal{Z}^p+pc_0\biggl(\frac{1}{X_0^{p+1}}+\frac{1}{X_4^{p+1}}\biggr) \mathcal{Z}^p
+p(c_0+\psi)\mathcal{Z}^{p-1}
=p(c_0+\psi)^p+H(X_0,X_4,\mathcal{Z})
\end{equation}
where $H(X_0,X_4,\mathcal{Z})$
 is a rigid analytic function
satisfying $v(H(X_0,X_4,\mathcal{Z}))>2-(p+1)v(X_0).$ 
Note that we have $v(\mathcal{Z})=1/p.$
We put $\mathcal{Z}_1:=\mathcal{Z}/(c_0+\psi).$
By dividing the above equality (\ref{err}) by $(c_0+\psi)^p,$ we obtain the following
\begin{equation}\label{an1}
\mathcal{Z}_1^p+pc_0\biggl(\frac{1}{X_0^{p+1}}+\frac{1}{X_4^{p+1}}\biggr) \mathcal{Z}_1^p
+p\mathcal{Z}_1^{p-1}
=p+H_1(X_0,X_4,\mathcal{Z}_1)
\end{equation}
where we write $H_1(X_0,X_4,\mathcal{Z}_1)$
for $H(X_0,X_4,\mathcal{Z}_1(c_0+\psi))/(c_0+\psi)^p.$
We choose a root $\gamma$ with $v(\gamma)=1/p$
such that
$\gamma^p+p\gamma^{p-1}=p.$
We introduce a new parameter $Z$ as follows
\begin{equation}\label{an2}
X_0X_4(X_0-X_4)^{p-1}/(c_0+\psi)=\mathcal{Z}_1=\gamma+Z.
\end{equation}
Substituting (\ref{an2}) to (\ref{an1}),
we obtain
\begin{equation}\label{an3}
Z^p+pc_0\gamma^p\biggl(\frac{1}{X_0^{p+1}}+\frac{1}{X_4^{p+1}}\biggr)=H_2(X_0,X_4,Z)
\end{equation}
where $H_2(X_0,X_4,\mathcal{Z})$
 is a rigid analytic function
satisfying $v(H(X_0,X_4,\mathcal{Z}))>2-(p+1)v(X_0).$ 
Here, we have $v(Z) \geq \frac{1}{p}\{2-(p+1)v(X_0)\}$ and hence $v(Z)>v(\gamma)$
by $v(X_0)<1/p(p+1).$
Again, we introduce a new parameter $Z_1$ as follows
$X_0-X_4=Z_1.$ 
Substituting $X_4=X_0-Z_1$ to (\ref{an2}), we acquire the following
\begin{equation}\label{rt}
X_0^2Z_1^{p-1}=c_0\gamma+g(X_0,Z,Z_1)
\end{equation}
with $v(g(X_0,Z,Z_1))>1/p.$ Thereby, we acquire
$2v(X_0)+(p-1)v(Z_1)=1/p$ by (\ref{rt}).
Note that we have $v(Z_1)>v(X_0).$
The equality (\ref{rt}) is decomposed to
$$X_0Z_1^{(p-1)/2}=\pm((c_0\gamma)^{1/2}+h(X_0,Z,Z_1))$$
where we have $v(h(X_0,Z,Z_1))>1/2p.$
We consider
\begin{equation}\label{an4'}
X_0Z_1^{(p-1)/2}=(c_0\gamma)^{1/2}+h(X_0,Z,Z_1).
\end{equation}
Substituting $X_4=X_0-Z_1$ to (\ref{an3}), we obtain the following
\begin{equation}\label{an4}
Z^p+2pc_0\gamma^p/X_0^{p+1}=H_3(X_0,Z,Z_1)
\end{equation}
where we write $v(H_3(X_0,Z,Z_1))>pv(Z).$
By (\ref{an4}), we obtain $v(Z)=(1/p)\{2-(p+1)v(X_0)\}.$
We choose a root
$\gamma_1$ satisfying $\gamma_1^p=-2pc_0\gamma^p.$ We have $v(\gamma_1)=2/p.$
We put $Z_2:=\gamma_1/X_0Z.$
We have $v(Z_2)=v(X_0)/p>0.$
Then, the equality (\ref{an4}) is rewritten as follows
$X_0=Z_2^p+H_4(X_0,Z_1,Z_2)$
with $v(H_4(X_0,Z_1,Z_2))>v(X_0).$
Eliminating $X_0$ from $H_4(X_0,Z_1,Z_2),$
this equality has the following form
\begin{equation}\label{an5}
X_0=Z_2^p+H_5(Z_1,Z_2)
\end{equation}
with some rigid analytic function $H_5(Z_1,Z_2),$
with $v(H_5(Z_1,Z_2))>v(X_0).$
Substituting this to (\ref{an4'}),
the following equality holds
\begin{equation}\label{an6}
Z_2^pZ_1^{(p-1)/2}=(c_0\gamma)^{1/2}+h(Z_1,Z_2).
\end{equation}
We write $Z_2=t^{(p-1)/2}.$
Substituting this to (\ref{an6}), we obtain
$(t^pZ_1)^{(p-1)/2}=(c_0\gamma)^{1/2}+h(Z_1,t).$
Hence we acquire
$t^pZ_1=(c_0\gamma)^{1/(p-1)}+h(Z_1,t).$
By this equality, we can rewrite this as follows
$Z_1=\frac{(c_0\gamma)^{1/(p-1)}}{t^p}+g(t).$
Therefore the parameters $Z_1,Z_2$
are written with respect to the parameter $t.$
Since we have $v(X_0)=(p(p-1)/2)v(t),$
we obtain the required width. 
For case $X_0Z_1^{(p-1)/2}=-((c_0\gamma)^{1/2}+h(X_0,Z,Z_1)),$
we can do the same argument as the one above.
Thus, the required assertions follow.
\end{proof}
\begin{proposition}\label{nan2}
Let the notation be as above.
\\1. We have the following descriptions
$$V_3(A) \cap U_2(A)=\{(E,C) \in W_A(p^4)|\frac{1}{p^2(p+1)}<h(E)<\frac{1}{2p^2},|K(E) \cap C|=p^3\},$$
$$V_2(A) \cap U_2(A)=\{(E,C) \in W_A(p^4)|\frac{1}{2p^2}<h(E)<\frac{1}{p(p+1)},|K(E) \cap C|=p^3\}.$$      
\\2. The intersections $V_3(A) \cap U_2(A)$ and $V_2(A) \cap U_2(A)$ are annuli of width
$\frac{(p-1)i(A)}{2p^3(p+1)}.$
\end{proposition}
\begin{proof} 
By the definitions of $V_3(A)$ and $U_2(A),$
we have $\pi_{1,2}(E,C) \in A(p^{-\frac{i(A)}{2p}},p^{-\frac{i(A)}{p(p+1)}})$
and $\pi_{2,1}(E,C) \in A(p^{-\frac{i(A)}{2}},p^{-\frac{i(A)}{p+1}}).$
By Theorem \ref{Ka}, we obtain the assertion 1 for $V_3(A) \cap U_2(A).$
For $V_2(A) \cap U_2(A),$ we prove in the same way. We omit the proof.

We prove the assertion 2. We prove the assertion only for $V_3(A) \cap U_2(A),$
because the argument is very similar.
By considering an embedding $\prod_{0 \leq i \leq 4} \pi_{i,4-i}$
as in Corollary \ref{prod},
we acquire
the following description of $V_3(A) \cap U_2(A)$  
$$\{(X_i)_{i=0}^{i=4} \in W_A(1)^{\times 5}|\ \frac{1}{p^2(p+1)}<v(X_0)<\frac{1}{2p^2},v(X_1)=pv(X_0),
v(X_2)=p^2v(X_0),v(X_2)+v(X_3)\\=1,$$
$v(X_3)=pv(X_4),F_p^{\beta_0}(X_i,X_{i+1})=0\ (0 \leq i \leq 3)\}.$
Then, we have $1/(p+1)<v(X_2)<1/2<v(X_3)<p/(p+1),$
$v(X_1)+v(X_4)=1/p, v(X_1)<v(X_4).$
$F_p^{\beta_0}(X_i,X_{i+1})=0\ (0 \leq i \leq 3)\}$  induce the followings
$$X_1=X_0^p+\frac{pc_0}{X_0}+pg(X_0), X_2=X_1^p+\frac{pc_0}{X_1}+pg(X_1),$$
$$X_3=X_2^p+\frac{pc_0}{X_2}+pg(X_2), X_3=X_4^p+\frac{pc_0}{X_4}+pg(X_4).$$
We rewrite
$X_3=X_2^p+\frac{pc_0}{X_2}+pg(X_2)$ as follows
$X_2X_3=pc_0+X_2^{p+1}+pX_2g(X_2).$
Substituting $X_2=X_1^p+(pc_0/X_1)+pg(X_1), X_3=X_4^p+(pc_0/X_4)+pg(X_4)$
to this equality, we obtain the following equality
\begin{equation}\label{bg}
X_1^pX_4^p+pc_0\frac{X_1^p}{X_4}+pc_0\frac{X_4^p}{X_1}=
pc_0+H(X_1,X_4)
\end{equation}
where $H(X_1,X_4)$ is a rigid analytic function satisfying 
$v(H(X_1,X_4))>((p^2-p-1)/p^2)+(p+1)^2v(X_0).$

We choose a root $\gamma$ of the following equation
$\gamma^p=pc_0.$ 
We introduce a new parameter $Z$ as follows
$X_1X_4=\gamma+Z$ with $v(Z)>v(\gamma)=1/p.$
By substituting $X_1X_4=\gamma+ Z$ to (\ref{bg}), we acquire the following equation
\begin{equation}\label{ci1}
Z^p+pc_0X_1^{p+1}/(\gamma+Z)+pc_0(\gamma+Z)^p/X_1^{p+1}=H_1(X_1,Z)
\end{equation}
with $v(H_1(X_1,Z))>((p^2-p-1)/p^2)+(p+1)^2v(X_0).$
Note that $v(Z)=(p-1)/p^2+(p+1)v(X_0)$
and hence $v(Z)>v(\gamma)$ by $v(X_0)>1/p^2(p+1).$
We choose a root $\gamma_1$ such that 
$\gamma_1^p=pc_0/\gamma$.
We introduce a new parameter $Z_1$
as follows
\begin{equation}\label{ci2}
Z+\gamma_1X_0^{p+1}+\gamma(\gamma+Z)/X_0^{p+1}=Z_1.
\end{equation}
By substituting $X_1=X_0^p+\frac{pc_0}{X_0}+pg(X_0)$
to the term $pc_0(\gamma+Z)^p/X_1^{p+1}$ in the
left hand side of the equality (\ref{ci1}),
we acquire the following 
\begin{equation}\label{tri}
\frac{pc_0(\gamma+Z)^p}{X_1^{p+1}}=\frac{pc_0(\gamma+Z)^p}{X_0^{p(p+1)}}+K(X_0,Z)
\end{equation}
where $v(K(X_0,Z))>\frac{p^2-p-1}{p^2}+(p+1)^2v(X_0).$
By substituting $X_1=X_0^p+\frac{pc_0}{X_0}+pg(X_0)$
to the term $pc_0X_1^{p+1}/(\gamma+Z)$ in the
left hand side of the equality (\ref{ci1}),
we acquire
\begin{equation}\label{tri2}
pc_0X_1^{p+1}/(\gamma+Z)=(pc_0/\gamma)X_0^{p(p+1)}-(pc_0/\gamma^2)X_0^{p(p+1)}Z+((pc_0)^2/\gamma)X_0^{p^2-1}
+G(X_0,Z)
\end{equation}
where $G(X_0,Z)$ is a rigid analytic function
whose valuation is the largest in the right hand side.
We have $v((pc_0)^2/\gamma)X_0^{p^2-1})>v((pc_0/\gamma^2)X_0^{p(p+1)}Z),$
 which is equivalent to $v(X_0)<(p^2+1)/2p^2(p+1).$
Hence, by (\ref{tri}) and (\ref{tri2}),
 we acquire the following equality by substituting (\ref{ci2}) to (\ref{ci1})
\begin{equation}\label{ci4}
Z_1^p=(pc_0/\gamma^2)X_0^{p(p+1)}Z+H_2(X_0,Z)
\end{equation}
with $v(H_2(X_0,Z))>v(Z_1^p).$
Note that $v(Z_1)=\frac{1}{p}(\frac{p^2-p-1}{p^2}+(p+1)^2v(X_0)).$
We can rewrite (\ref{ci2}) as follows
$Z=-\gamma_1X_0^{p+1}+h(X_0,Z_1)$ with $v(h(X_0,Z_1))>v(Z).$
Substituting this to the right hand side of (\ref{ci4}), we obtain the following equality
\begin{equation}\label{ci3}
Z_1^p=-(pc_0/\gamma^2)\gamma_1X_0^{(p+1)^2}+h_1(X_0,Z_1)
\end{equation}
with $v(h_1(X_0,Z_1))>v(Z_1^p).$
We choose a root $\gamma_3$ of an equation
$X^p=-(pc_0/\gamma^2)\gamma_1.$ 
We have $v(\gamma_3)=\frac{p^2-p-1}{p^2}.$
We introduce a new parameter $Z_2$ as follows
$Z_2:=Z_1/\gamma_3X_0^{p+2}.$
We have $v(Z_2)>0.$
The equality (\ref{ci3}) is rewritten as follows 
$Z_2^p=X_0+h_2(X_0,Z_2)$
with $v(h_2(X_0,Z_2))>v(X_0).$
Further, this equality can be written as follows
$X_0=Z_2^p+h_3(Z_2)$ with some rigid analytic function $h_3.$
Therefore $X_0,Z_1,Z$ are written with respect to the parameter $Z_2.$
Hence the space $V_3(A) \cap U_2(A)$
is an annulus.
Since we have $1/p^2(p+1)<v(Z_2)<1/2p^3$, we also obtain the required width.
Thus, the assertions have been proved.
\end{proof}

Let $a,b>0$ be positive integers satisfying $a+b=4.$
If $a \leq b,$
 the reduction
of the underlying affinoid $\mathbf{X}^{\pm}_{a,b} \subset W^{\pm}_{a,b}$
is isomorphic to the Igusa curve
$Ig(p^a)$
as mentioned above.
 The reduction of $W_{a,b}$ with $ab=0$
is isomorphic to $\mathbb{P}^1.$
These facts are proved in \cite[Corollary 3.7]{CM}.

We are ready to state and prove our main theorem in this paper.
\begin{theorem}\label{otto}
Let $p \geq 13$ be a prime.
The covering $\mathcal{C}_0(p^4)$ of $X_0(p^4),$ which is made up of
$$\{W^{\pm}_{a,b}|a,b \geq 0, a+b=4\}$$
and the union over all supersingular curves $A$ of
$$\{V_1(A),V_3(A),\hat{V}_2(A),\hat{U}_1(A),\hat{U}_2(A)\} \cup
\mathcal{S}_1(A) \cup \mathcal{S}_2(A) \cup \mathcal{T}(A),$$
is stable (over $\mathbb{C}_p$).
\end{theorem}
\begin{proof}
We imitate the arguments in \cite[the proof of Theorem 9.2]{CM}.
We know that the elements of $\mathcal{C}_0(p^4)$
are wide opens, and that $(S,\mathbf{X}_S)$
is a basic wide open pair for each $S \in \mathcal{S}_i(A),\mathcal{T}(A).$
The $W^{\pm}_{a,b}$
are disjoint from each other, and $W^{\pm}_{a,b}$
intersects $W_A(p^4)$
only at $V_2(A)$
when $a>b,$ only at $V_2(A)$ when $a=b=2$
and only at $V_3(A)$
when $a<b.$
The bulk
of what we have to show is that whenever two wide opens in the cover do intersect,
the intersection is the disjoint union of annuli.
Then we have to show that each wide open is basic, with an underlying affinoid 
that has good reduction.
The required assertion follows from 
Lemmas \ref{Lee1} and \ref{Lee2},
 Propositions \ref{nan1} and \ref{nan2}, and the comuputations in Section 4.
\end{proof}
We have the following partial dual graph of the stable reduction of $X_0(p^4)$

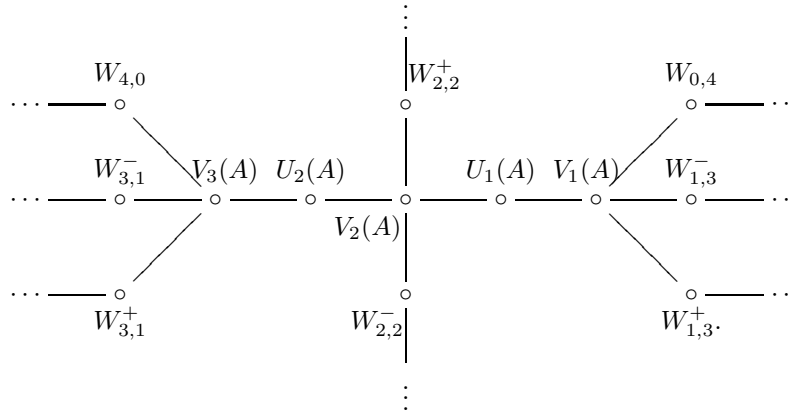
\centerline{
\xygraph{
\circ ([]!{+(0,+.3)} {W^{-}_{3,1}})(-[l]\cdots)(-[r]
\circ ([]!{+(+.1,+.3)} {V_3(A)})(-[dl] \circ ([]!{+(0,-.3)} {W^+_{3,1}})(-[l] \cdots))(-[ul] \circ
([]!{+(0,+.3)} {W_{4,0}})(-[l] \cdots))(-[r]
\circ ([]!{+(0,+.3)} {U_2(A)})-[r]
\circ ([]!{-(+.4,+.3)} {V_2(A)})(-[d] \circ ([]!{+(-.3,-.3)} {W^{-}_{2,2}})(-[d] \vdots))
(-[u] \circ([]!{+(+.3,+.3)} {W^{+}_{2,2}})(-[u] \vdots))(-[r]
\circ ([]!{+(0,+.3)} {U_1(A)})
(-[r]
\circ([]!{+(-.1,+.3)} {V_1(A)})(-[r] \circ([]!{+(0,+.3)} {W^{-}_{1,3}})(-[r] \cdots))
(-[ur] \circ([]!{+(0,+.3)} {W_{0,4}})(-[r] \cdots)) 
(-[dr] \circ([]!{+(0,-.3)} {W^{+}_{1,3}.})(-[r] \cdots))))))}}

We obtain the following corollary as an immediate consequence of Theorem \ref{otto}.
\begin{corollary}\label{jac}
Let $J_0(p^n)$
denote the Jacobian of the modular curve $X_0(p^n).$
Let $ss$ be the number of the supersingular elliptic curves on $X_0(1).$
Then, the toric rank of the Jacobian $J_0(p^4)$
is given by $7(ss-1).$
\end{corollary}
\begin{proof}
In general, let $C$ be a projective smooth and geometrically connected curve over a local field $K,$
 and $\mathcal{C}/\mathcal{O}_K$ its stable model. 
Then, the toric rank of the Jacobian $J(C)$ of the curve $C$
is known to be equal to the Betti number of the dual graph of the stable reduction
$\overline{\mathcal{C}}.$
Hence, the required assertion follows from Theorem \ref{otto}.
\end{proof}
\begin{remark}
In general, let $C$ be a projective smooth and geometrically connected curve over a local field $K,$
 and $\mathcal{C}/\mathcal{O}_K$ its stable model.
We assume that all irreducible components in the stable reduction
$\overline{\mathcal{C}}$
are smooth over $\mathbb{F}_K.$
Then, the abelian part of the ${\rm N\acute{e}ron}$
model of the Jacobian of $C$ is isomorphic to the product of the Jacobians 
of each irreducible component in the stable reduction $\overline{\mathcal{C}}.$ 
\end{remark}

\subsection{Genus Calculation in \cite[subsection 5.3]{Mc}}
For the convenience  of a reader, we write down a genus computation of $X_0(p^4)$
already checked in \cite[subsection 5.3]{Mc}.
Let $g_0(p^4)$ be the genus of the modular curve $X_0(p^4)$
and $b(\Gamma)$ the Betti number of the dual graph $\Gamma$
of the stable model of $X_0(p^4).$
In this subsection, we will check the following equality
\begin{equation}\label{q5.3}
g_0(p^4)=\sum_{C:\ {\rm irred.\ comp.}}g(C)+b(\Gamma)
\end{equation}
Let $g_e$ denote the genus of the Igusa curve ${\rm Ig}(p^e)$ for $p>3.$
Then, we have by \cite[Section 0]{Ig}
$$2g_e-2=\frac{1}{24}(p-1)(p^{2e-1}-12p^{e-1}+1)-ss$$
where $ss$ is the numer of the supersingular points including $j=0$
and/or $j=1728$
as in Corollary \ref{jac}.

For simplicity, we deal with the case where $p \equiv 11\ ({\rm mod}\ 12).$
In this case, the genera of ${\rm Ig}(p)$
and ${\rm Ig}(p^2)$
are given by
$$g_1=\frac{p^2-14p+33}{48}$$
$$g_2=\frac{p^4-p^3-12p^2+11p+21}{48}.$$
There are $ss=\frac{p+13}{12}$
supersingular elliptic curves including $j=0$ and $j=1728.$
Thus, we acquire the following equality
$b(\Gamma)=7(ss-1)=7(p+1)/12$
as in Corollary \ref{jac}.
We collect the data for each special affinoid in Table below.
\begin{center}
{\renewcommand\arraystretch{2}
\begin{tabular}{|c|c|c|c|}
\hline  & ${\mathbf{Z}}^A_{1,2},
{\mathbf{Z}}^A_{2,1}$ & ${\mathbf{Y}}^A_{1,3}, {\mathbf{Y}}^A_{3,1}$ 
& ${\mathbf{Y}}^A_{2,2}$ \\ 
\hline  $j(A)=0$ & $\frac{p^2-1}{3}$ & $\frac{p-5}{6}$ & 
$\frac{p^3-5}{6}$ \\ 
\hline  $j(A)=1728$ & $\frac{p^2-1}{2}$ & $\frac{p-3}{4}$ &  $\frac{p^3-3}{4}$\\ 
\hline  otherwise & $p^2-1$ & $\frac{p-1}{2}$ &  $\frac{p^3-1}{2}$\\ 
\hline 
\end{tabular}}
\end{center} 
\begin{center}
Table 2: Total Genera of Supersingular Components of $X_0(p^4)$
\end{center}
Therefore, for the sum of the genera of all components in the reduction
of any given $W_A(p^4),$
we obtain the following
$$G_0=\frac{p^3+4p^2+2p-19}{6}\ (j(A)=0)$$
$$G_{1728}=\frac{p^3+4p^2+2p-13}{4}\ (j(A)=1728)$$
$$G_A=\frac{p^3+4p^2+2p-7}{2}\ (j(A) \neq 0,1728)$$
where $G_A$ is the sum of the genera of the irreducible components which
appear in the reduction of the supersingular locus $W_A(p^4).$
By summing the Betti number and the genera of all components, we acquire the following
$$\frac{7(p+1)}{12}+4g_1+2g_2+G_0+G_{1728}+\frac{p-11}{12}G_A
=\frac{p^3(p+1)}{12}-\frac{p(p+1)}{12}+1.$$
The right hand side is equal to the genus $g_0(p^4)$
by \cite[Proposition 1.40 and 1.43]{Sh}.
Hence, we have checked the equality (\ref{q5.3}).

\subsection{Intersection Data}
We include the intersection multiplicities 
in $X_0(p^n)\ (n=2,3,4)$
below in Tables 3, 4 and 5.
These numbers have been obtained via a rigid analytic reformulation.
In particular, suppose that $X$ and $Y$ are irreducible components
of a curve with semi-stable reduction over some extension $K/\mathbb{Q}_p,$
and that they intersect in an ordinary double point $P.$
Then ${\rm red}^{-1}(P)$
is an annulus, say with width $w(P).$
Let $e_p(K)$ denote the ramification index of $K/\mathbb{Q}_p.$
In this case, the intersection multiplicity
of $X$ and $Y$ at $P$ can be found
by
$$M_K(P):=e_p(K) \cdot w(P).$$
Note that while intersection multiplicity depends on $K,$
the width makes sense even over $\mathbb{C}_p,$
which in some sense makes width a more natural invariant from
the purely geometric perspective as mentioned in
\cite[Section 9.1]{CM}.
Now, for our calculation of $M_K(P)$ on $X_0(p^2),$
we take $e_p(K)=(p^2-1)/2.$
See also \cite[subsection 2.3.2]{E2}.

\begin{center}
\begin{minipage}{.4\textwidth}
\begin{tabular}{|c|c|c|}
\hline  $P$ & \scriptsize{$(\mathbf{X}_{2,0},\mathbf{Y}^A_{1,1}),(\mathbf{X}_{0,2},\mathbf{Y}^A_{1,1})$} & 
\scriptsize{$(\mathbf{X}^{\pm}_{1,1},\mathbf{Y}^A_{1,1})$}  \\ 
\hline
\rule[-13pt]{0pt}{31pt}
$w(P)$ & $\frac{i(A)}{p+1}$ & $\frac{2 \cdot i(A)}{p^2-1}$ \\ 
\hline \rule[-13pt]{0pt}{31pt} $M_K(P)$ & $\frac{p-1}{2} \cdot i(A)$ & $i(A)$\\ 
\hline 
\end{tabular}
\end{minipage}
\end{center}

\begin{center}
Table 3: Intersection Multiplicity Data for $X_0(p^2)$ 
\end{center}

Now, for our calculation of $M_K(P)$ on $X_0(p^3),$
we take $e_p(K)=p^2(p^2-1).$
See also \cite[Table 1 in Section 9]{CM}.

\begin{minipage}{.4\textwidth}
\begin{tabular}{|c|c|c|c|c|}
\hline  $P$ & \scriptsize{$(\mathbf{X}_{3,0},\mathbf{Y}^A_{2,1}),(\mathbf{X}_{0,3},\mathbf{Y}^A_{1,2})$} & 
\scriptsize{$(\mathbf{X}^{\pm}_{2,1},\mathbf{Y}^A_{2,1}),(\mathbf{X}^{\pm}_{1,2},\mathbf{Y}^A_{1,2})$} & 
\scriptsize{$(\mathbf{Y}^{A}_{2,1},\mathbf{Z}^A_{1,1}), 
(\mathbf{Y}^A_{1,2},\mathbf{Z}^A_{1,1})$} 
& \scriptsize{$(\mathbf{X}_S,\mathbf{Z}^A_{1,1})\ (S \in \mathcal{S}(A))$} \\ 
\hline
\rule[-13pt]{0pt}{31pt}
$w(P)$ & $\frac{i(A)}{p(p+1)}$ & $\frac{2 \cdot i(A)}{p(p^2-1)}$ & 
$\frac{(p-1) \cdot i(A)}{2p^2(p+1)}$ & $\frac{i(A)}{4p^2}$\\ 
\hline \rule[-13pt]{0pt}{31pt} $M_K(P)$ & $p(p-1) \cdot i(A)$ & $2p \cdot i(A)$ &  
$\frac{(p-1)^2}{2} \cdot i(A)$ & $\frac{p^2-1}{4} \cdot i(A)$\\ 
\hline 
\end{tabular}
\end{minipage}

\begin{center}
Table 4: Intersection Multiplicity Data for $X_0(p^3)$ 
\end{center}

Now, for our calculation of $M_K(P)$ on $X_0(p^4),$
we take $e_p(K)=p^3(p^2-1).$

\begin{minipage}{.4\textwidth}
\begin{tabular}{|c|c|c|c|c|}
\hline  $P$ & \scriptsize{$(\mathbf{X}_{4,0},\mathbf{Y}^A_{3,1}),(\mathbf{X}_{0,4},\mathbf{Y}^A_{1,3})$} & 
\scriptsize{$(\mathbf{X}^{\pm}_{3,1},\mathbf{Y}^A_{3,1}),(\mathbf{X}^{\pm}_{1,3},\mathbf{Y}^A_{1,3})$} & 
\scriptsize{$(\mathbf{X}^{\pm}_{2,2},\mathbf{Y}^A_{2,2})$} & 
\scriptsize{$(\mathbf{X}_T,\mathbf{Y}^A_{2,2})\ (T \in \mathcal{T}(A))$} \\ 
\hline
\rule[-13pt]{0pt}{31pt}
$w(P)$ & $\frac{i(A)}{p^2(p+1)}$ & $\frac{2 \cdot i(A)}{p^2(p^2-1)}$ & 
$\frac{2 \cdot i(A)}{p^2(p^2-1)}$ & $\frac{i(A)}{p^2(p+1)}$\\ 
\hline \rule[-13pt]{0pt}{31pt} $M_K(P)$ & $p(p-1) \cdot i(A)$ & $2p \cdot i(A)$ &  
$2p \cdot i(A)$ & $p(p-1) \cdot i(A)$\\ 
\hline 
\end{tabular}
\end{minipage}

\begin{minipage}{.4\textwidth}
\begin{tabular}{|c|c|c|}
\hline \scriptsize{$(\mathbf{Z}^A_{2,1},\mathbf{Y}^A_{3,1}),(\mathbf{Z}^A_{1,2},\mathbf{Y}^A_{1,3})$}&
\scriptsize{$(\mathbf{Z}^A_{2,1},\mathbf{Y}^A_{2,2}),(\mathbf{Z}^A_{1,2},\mathbf{Y}^A_{2,2})$} &
\scriptsize{$(\mathbf{Z}^A_{2,1},\mathbf{X}_S)\ 
(S \in \mathcal{S}_2(A)),(\mathbf{Z}^A_{1,2},\mathbf{X}_S)\ (S \in \mathcal{S}_1(A))$} \\
\hline
\rule[-13pt]{0pt}{31pt}
 $\frac{(p-1) \cdot i(A)}{2p^3(p+1)}$ & $\frac{(p-1) \cdot i(A)}{2p^3(p+1)}$ & $\frac{i(A)}{4p^3}$\\
\hline \rule[-13pt]{0pt}{31pt}$\frac{(p-1)^2 \cdot i(A)}{2}$ &  $\frac{(p-1)^2 \cdot i(A)}{2}$  & $\frac{(p^2-1) \cdot i(A)}{4}$\\
\hline
\end{tabular}
\end{minipage}

\begin{center}
Table 5: Intersection Multiplicity Data for $X_0(p^4)$ 
\end{center}

\end{document}